\documentclass[11pt]{article}
\usepackage{amsmath, amscd, amssymb, latexsym, epsfig, color, amsthm}
\numberwithin{equation}{section}

\def\proof{\smallskip\noindent {\it Proof: \ }}
\def\endproof{\hfill$\square$\medskip}
\newtheorem{theorem}{Theorem}[section]
\newtheorem{proposition}[theorem]{Proposition}

\newtheorem{corollary}[theorem]{Corollary}

\newtheorem{lemma}[theorem]{Lemma}

\newtheorem{problem}[theorem]{Problem}

\theoremstyle{definition}
\newtheorem{definition}[theorem]{Definition}
\newtheorem{example}[theorem]{Example}
\newtheorem{remark}[theorem]{Remark}

\DeclareMathOperator\lk{\mathrm{lk}}

\DeclareMathOperator\st{\mathrm{st}}
\DeclareMathOperator\cost{\mathrm{cost}}

\DeclareMathOperator{\Hilb}{\mathrm{Hilb}}
\DeclareMathOperator{\Soc}{\mathrm{Soc}}
\DeclareMathOperator\depth{\mathrm{depth}}

\newcommand{\complet}[1]{\hat{#1}}

\newcommand{\SB}{\mathcal{SB}}

\newcommand{\field}{\mathbf{k}}

\newcommand{\R}{{\mathbb R}}
\newcommand{\Q}{{\mathbb Q}}
\newcommand{\Z}{{\mathbb Z}}
\newcommand{\ZZ}{{\mathbb Z}}

\newcommand{\C}{{\mathcal C}}
\newcommand{\K}{{\mathcal K}}
\newcommand{\T}{{\mathcal T}}
\newcommand{\Sm}{{\mathcal S}}
\newcommand{\M}{{\mathcal M}}
\newcommand{\HH}{{\mathcal H}}

\newcommand{\mideal}{\ensuremath{\mathfrak{m}}}

\newcommand{\Ker}{\ensuremath{\mathrm{Ker}}\hspace{1pt}}
\newcommand{\Coker}{\ensuremath{\mathrm{Coker}}\hspace{1pt}}

\newcommand{\halffloor}{\lfloor \frac{d}{2}\rfloor}

\title{$g$-vectors of manifolds with boundary}
\author{Isabella Novik\thanks{Research is partially\textsl{}
supported by NSF grant DMS-1664865 and by Robert R.~\& Elaine F.~Phelps Professorship in Mathematics}\\
\small Department of Mathematics\\[-0.8ex]
\small University of Washington\\[-0.8ex]
\small Seattle, WA 98195-4350, USA\\[-0.8ex]
\small \texttt{novik@math.washington.edu}
\and Ed Swartz\\
\small Department of Mathematics \\[-0.8ex]
\small Cornell University \\[-0.8ex] 
\small Ithaca, NY 14853-4201, USA \\[-0.8ex]
\small \texttt{ebs22@cornell.edu } }
\begin{document}
\maketitle

\begin{abstract} 
We extend several $g$-type theorems for connected, orientable homology manifolds without boundary to manifolds with boundary. As applications of these results we obtain K\"uhnel-type bounds on the Betti numbers as well as on  certain weighted sums of Betti numbers of manifolds with boundary. Our main tool is the completion $\complet\Delta$ of a manifold with boundary $\Delta$; it is obtained from $\Delta$ by coning off the boundary of $\Delta$ with a single new vertex. We show that despite the fact that $\complet{\Delta}$ has a singular vertex, its Stanley--Reisner ring shares a few properties with the Stanley--Reisner rings of  homology spheres.   We close with a discussion of a connection between three lower bound theorems for manifolds,  PL-handle decompositions, and surgery.

\end{abstract}

\section{Introduction} \label{sec:Intro}

This paper is devoted to the study of face numbers of manifolds with boundary. While \cite{Murai-Novik-bdry} established the best to-date lower bounds on the $g$-numbers of manifolds with boundary, our emphasis here is on {\em Macaulay-type} inequalities involving the $g$-numbers.

The quest for characterizing possible $f$-vectors of various classes of simplicial complexes or at least establishing significant necessary conditions started about fifty years ago with McMullen's $g$-conjecture \cite{McMullen-71} that posited a complete characterization of $f$-vectors of simplicial polytopes. In ten years, this conjecture became a theorem \cite{Billera-Lee-81, Stanley-80}. This gave rise to algebraic and combinatorial versions of the $g$-conjecture for simplicial spheres. Very recently Adiprasito \cite{Adiprasito-g} announced a proof of the most optimistic algebraic version of this conjecture. In the late 1990s, Kalai proposed a far reaching generalization of the sphere $g$-conjecture to simplicial manifolds without boundary.  The authors proved that the (weaker) algebraic version of the $g$-conjecture for spheres implies all the enumerative consequences of Kalai's manifold $g$-conjecture, see \cite{Novik-Swartz-09:Gorenstein}. Furthermore, Murai and Nevo \cite{Murai-Nevo-14} establsihed a $\tilde{g}$-variation of this result. In this paper we extend both of these statements to manifolds with boundary.

The main idea (that goes back to Kalai \cite[Section 10]{Kalai-87}) is as follows: given a simplicial complex $\Delta$ whose geometric realization is a connected, orientable, homology manifold with boundary, we define the completion of $\Delta$ --- a complex $\complet{\Delta}$ obtained from $\Delta$ by coning the boundary of $\Delta$ (all components of it) with a single new vertex $v_0$. We then show that, despite the fact that $\complet{\Delta}$ has a singular vertex,  a certain quotient of a generic Artinian reduction of the Stanley--Reisner ring of $\complet{\Delta}$ enjoys several properties that Artinian reductions of the Stanley--Reisner rings of simplicial spheres have. This result together with the computation of the Hilbert function of this quotient allows us to extend virtually all known results on face numbers of orientable manifolds without boundary to the class of orientable manifolds with boundary.

The main results and the structure of the paper are as follows. In Section 2 we discuss basics of simplicial complexes and Stanley--Reisner rings. In particular, we review Gr\"abe's theorem on local cohomology \cite{Grabe-84} and introduce our main player --- the completion $\complet{\Delta}$ of a manifold with boundary $\Delta$. Section 3 is devoted to the Gorensteiness and the weak Lefschetz property of a certain quotient of the Stanley--Reisner ring of $\complet{\Delta}$, see Theorem \ref{thm-Gorenstein} and Corollary \ref{WLP}. Section 4 computes the Hilbert function of this quotient, Theorem \ref{soc-and-h''}. This result is used in Section 5 to establish two versions of $g$-theorems for manifolds with boundary, Theorems \ref{g-theorem-v1} and \ref{g-theorem-v2}. In Section 6 we use these $g$-results to derive K\"uhnel-type bounds on the Betti numbers and certain weighted sums of Betti numbers of manifolds with boundary. Finally, in Section 7, we examine the combinatorial and topological consequences of some of the known inequalities for $f$-vectors of homology manifolds with boundary when they are sharp. More specifically, we discuss a connection between three lower bound theorems for manifolds, PL-handle decompositions, and surgery, see Theorems \ref{h''=0 -> handle decomposition}, \ref{global minimal g2}, and \ref{minimal-tilde-g}.

\section{Preliminaries} \label{sec:Prelim}
In this section we review the necessary background material on simplicial complexes and their Stanley--Reisner rings with a special emphasis on homology manifolds with and without boundary as well as on singular simplicial complexes that have only one singular vertex. We refer the reader to  \cite[Chapter 2]{Stanley-96} and the papers \cite{Novik-Swartz-09:Socles, Novik-Swartz-12} for more details on the subject. 

\subsection{Simplicial complexes: homology manifolds and their completions}
A {\em simplicial complex} $\Delta$ on the vertex set $V$ is a collection of subsets of $V$ that is closed under inclusion and contains all singletons $\{v\}$ for $v\in V$. The elements of $\Delta$ are called {\em faces}. The maximal faces (with respect to inclusion) are called {\em facets}. The {\em dimension of a face} $F\in \Delta$ is $\dim F:=|F|-1$ and
the {\em dimension of $\Delta$} is the maximal dimension of its faces. A complex is {\em pure} if all of its facets have the same dimension. A complex $\Delta$ is {\em $j$-neighborly} if every $j$-element subset of $V$ is a face of $\Delta$. 


Let $\Delta$ be a simplicial complex and let $F$ be a face of $\Delta$. The {\em star} and the {\em link} of $F$ in $\Delta$ are the following subcomplexes of~$\Delta$: 
$$
\st F=\st_\Delta F:=\{G\in \Delta \, \mid \, G\cup F\in \Delta\}, \quad
 \lk F=\lk_\Delta F:=\{G\in \st_\Delta F \, \mid \, G\cap F=\emptyset\}.
$$
In particular, the link of the empty face is the complex $\Delta$ itself. We refer to the links of non-empty faces as {\em proper links}. The {\em contrastar} of $F$ in $\Delta$ (also known as the {\em deletion} of $F$ from $\Delta$) is $\cost F=\cost_\Delta F:=\{G\in \Delta \mid G\not\supseteq F\}$.  If $v$ is a vertex, it is customary to write $v\in\Delta$, $\st v$, $\lk v$, and $\cost v$ instead of $\{v\}\in\Delta$, $\st \{v\}$, $\lk\{v\}$, and $\cost\{v\}$. (In fact, we will sometimes write $\Delta-v$ instead of $\cost\{v\}$.)
Also, if $v_0\notin V$ is a new vertex, then the {\em cone} over $\Delta$ with apex $v_0$ is $v_0\ast\Delta:=\Delta \cup \{v_0\cup F \mid F\in \Delta\}$.

Throughout the paper, we fix an {\em infinite} field $\field$. We denote by $\tilde{H}_\ast(\Delta;\field)$ the {\em reduced simplicial homology} of $\Delta$ with coefficients in $\field$ and by $\tilde{\beta}_i(\Delta):=\dim_\field \tilde{H}_i(\Delta;\field)$ the $i$-th {\em reduced Betti number}.

One of the central objects of this paper is a {\em $\field$-homology manifold}. A pure $(d-1)$-dimensional simplicial complex  $\Delta$ is a {\em $\field$-homology manifold without boundary} (or a {\em closed $\field$-homology manifold}) if the homology (computed over $\field$) of every proper link of $\Delta$, $\lk_\Delta F$, coincides with the homology of a $(d-1-|F|)$-dimensional sphere. In this case, we write $\partial \Delta = \emptyset$. Similarly, a pure $(d-1)$-dimensional simplicial complex  $\Delta$ is a {\em $\field$-homology manifold with boundary} if every proper link of $\Delta$,  $\lk_\Delta F$, has the homology of a $(d-1-|F|)$-dimensional sphere or a ball (over $\field$), and the {\em boundary complex} of $\Delta$, i.e., 
\[ \partial \Delta:= \left\{F \in \Delta \mid \tilde{H}_\ast(\lk_\Delta F; \field)=0\right\}\cup\{\emptyset\}, \]
is a $(d-2)$-dimensional $\field$-homology manifold without boundary. The faces of $\partial\Delta$ are called {\em boundary faces}. The non-boundary faces of $\Delta$ are called {\em interior faces}.  When the field plays no role we simply call $\Delta$ a homology manifold (with or without boundary).

The prototypical example of a homology manifold (with or without boundary) is a triangulation of a topological manifold (with or without boundary). A connected $\field$-homology manifold $\Delta$ is {\em orientable} if the top homology of the pair $(\Delta, \partial \Delta)$ is $1$-dimensional. In this case,  $(\Delta, \partial \Delta)$ satisfies the usual Poincar\'e--Lefschetz duality associated with orientable compact manifolds. Note that an arbitrary triangulation of any topological manifold (orientable or not, with or without boundary) is an {\em orientable} $\Z/2\Z$-homology manifold.

A $\field$-homology $(d-1)$-sphere is a $(d-1)$-dimensional $\field$-homology manifold without boundary that has the same homology as the $(d-1)$-dimensional sphere. A $\field$-homology $(d-1)$-ball is a $(d-1)$-dimensional $\field$-homology manifold with boundary whose homology is trivial and whose boundary complex is a $\field$-homology $(d-2)$-sphere.  The contrastar of any vertex in a $\field$-homology $(d-1)$-sphere is a $\field$-homology $(d-1)$-ball. Furthermore, every proper link of a $\field$-homology manifold without boundary is a $\field$-homology sphere, while a proper link of a $\field$-homology manifold with boundary is a $\field$-homology sphere or  ball.

Let $\Delta$ be a $\field$-homology manifold with or without boundary and let $v_0\not\in V$ be a new vertex. A key to most of our proofs is the {\em completion of $\Delta$}, $\complet{\Delta}$, defined as follows:
\[
\complet{\Delta}:=\Delta\cup (v_0\ast\partial\Delta).
\]
 Note that we define $v_0\ast\emptyset =\emptyset$; hence if $\Delta$ is a homology manifold without boundary, then $\complet{\Delta}=\Delta$. 

A pure simplicial complex $\Gamma$ is a {\em complex with at most one singularity} if all of the vertex links of $\Gamma$ but possibly one are $\field$-homology balls or spheres. This exceptional vertex is called a {\em singular} vertex; the other vertices are called {\em non-singular}. For instance, if $\Delta$ is a $\field$-homology manifold with boundary, $\complet{\Delta}$ is a completion of $\Delta$, and $v\neq v_0$, then both $\complet{\Delta}$ and $\cost_{\complet{\Delta}}v$ are complexes with (at most) one singular vertex, namely, $v_0$.

When only topological properties of a space are relevant we may use capital roman letters.  For instance, ``If $X$ is a $d$-dimensional ball, then its boundary $Y$ is a $(d-1)$-dimensional sphere.''


\subsection{Face numbers and the Stanley--Reisner rings} Let $\Delta$ be a $(d-1)$-dimensional simplicial complex on the vertex set $V$. Denote by $f_i(\Delta)$ the number of $i$-dimensional faces of $\Delta$.
The {\em $f$-vector} of $\Delta$ is $f(\Delta)=(f_{-1}(\Delta), f_0(\Delta),\ldots, f_{d-1}(\Delta))$ and the {\em $h$-vector} of $\Delta$ is $h(\Delta)=(h_0(\Delta), h_1(\Delta), \ldots, h_d(\Delta))$, where 
\[h_i(\Delta):=\sum_{j=0}^i (-1)^{i-j}\binom{d-j}{d-i}f_{j-1}(\Delta).\]

Let $A=\field[x_v \mid v\in V]$ be a polynomial ring, and let $\mideal=(x_v \mid v\in V)$ be the graded maximal ideal of $A$. For $F\subseteq V$, write $x_F:=\prod_{v\in F} x_v$. The {\em Stanley--Reisner ideal} $I_\Delta$ of $\Delta$ is the ideal of $A$ defined by 
\[
I_\Delta=(x_F \mid F\subseteq V, F\notin\Delta).\]
The {\em Stanley--Reisner ring} $\field[\Delta]$ of $\Delta$ (over $\field$) is the quotient ring $\field[\Delta]=A/I_\Delta$. In particular, $\field[\Delta]$ is a graded ring; it is also a graded $A$-module. If $\dim\Delta=d-1$, then the Krull dimension of $\field[\Delta]$ is $d$ and the Hilbert series of $\field[\Delta]$ is given by   (\cite[Chapter II.1]{Stanley-96})
\[
\Hilb(\field[\Delta]; \lambda)=\frac{\sum_{i=0}^d h_i(\Delta)\lambda^i}{(1-\lambda)^d}.\]

A {\em linear system of parameters} (or {\em l.s.o.p} for short) for $\field[\Delta]$ is a sequence $\Theta=\theta_1,\ldots,\theta_d$ of $d=\dim\Delta + 1$ linear forms in $\mideal$ such that \[\field(\Delta,\Theta):= \field[\Delta]/\Theta\field[\Delta]\]
has Krull dimension zero (i.e., it is a finite-dimensional $\field$-space). Since $\field$ is infinite,  by the Noether normalization lemma, an l.s.o.p.~always exists: a generic choice of $\theta_1,\ldots,\theta_d$ does the job.
The ring $\field(\Delta,\Theta)$ is called an {\em Artinian reduction} of $\field[\Delta]$. 

We need a few more definitions. If $M$ is a finitely-generated graded $A$-module,  we let $M_j$ denote the $j$-th homogeneous component of $M$. For $\tau\in A$, define $(0:_M \tau):=\{\nu\in M \mid \tau\nu=0\}$. The {\em socle} of $M$ is the following graded submodule of $M$:
\[\Soc M=\bigcap_{v\in V}(0:_M x_v)=\{\nu\in M \mid \mideal\nu=0\}.\]
In particular, for any choice of integers $i_1<i_2<\cdots<i_\ell$, $\bigoplus_{j=1}^\ell (\Soc M)_{i_j}$ is a submodule of $M$.
For a standard graded $\field$-algebra $M=A/I$ of Krull dimension zero, this allows us to define the {\em interior socle} of $M$: 
\[\Soc^\circ M:=\bigoplus_{i=0}^{d_0-1}(\Soc M)_i, \quad \mbox{where $d_0:=\max\{j \mid M_j\neq 0\}$}.\]
We say that $A/I$ is a {\em level algebra} if $\Soc^\circ(A/I)=0$, and that $A/I$ is {\em Gorenstein} if it has a $1$-dimensional socle. Equivalently, $A/I$ is Gorenstein if it is level and $\dim_\field (A/I)_{d_0}=1$.

We are interested in the Hilbert functions of $\field(\Delta,\Theta)$ and its quotient \[\overline{\field(\Delta,\Theta)}:= \field(\Delta,\Theta)/\Soc^\circ \field(\Delta,\Theta).\] 
\begin{definition}  \label{def-h'-h''}
Let $\Delta$ be a $(d-1)$-dimensional simplicial complex and let $\Theta=\theta_1,\ldots,\theta_d$ be a {\em generic} l.s.o.p.~for $\field[\Delta]$. The {\em $h'$-} and the {\em $h''$-numbers} of $\Delta$ are defined by
\[
h'_j(\Delta):=\dim_\field \field(\Delta,\Theta)_j \quad \mbox{and} \quad h''_j(\Delta):=\dim_\field \overline{\field(\Delta,\Theta)}_j \quad \mbox{(for  $j\geq 0$), respectively.}
\]
\end{definition} 

Although $\field$ is suppressed from our notation, the $h'$- and $h''$-numbers do depend on $\field$. For any $(d-1)$-dimensional simplicial complex $\Delta$, $h'_j(\Delta)=0$ for all $j>d$ (see \cite[Proposition III.2.4(b)]{Stanley-96}), while $h'_d(\Delta)=h''_d(\Delta)=\tilde{\beta}_{d-1}(\Delta)$ (see \cite[Theorem 4.1]{TayWhiteley} and \cite[Lemma 2.2(3)]{Babson-Novik}). The following theorem collects several other known results on $h'$- and $h''$-numbers.

\begin{theorem} \label{known-h'-h''} Let $\Delta$ be a $(d-1)$-dimensional simplicial complex.
\begin{enumerate}
\item If $\Delta$ is a $\field$-homology sphere or a ball, then $h'_i(\Delta)=h_i(\Delta)$ for all $0\leq i\leq d$.
\item If $\Delta$ is a $\field$-homology manifold with or without boundary, then 
\[h'_i(\Delta)=h_i(\Delta)-\binom{d}{i}\sum_{j=1}^{i-1}(-1)^{i-j}\tilde{\beta}_{j-1}(\Delta) \qquad \forall ~0\leq i\leq d.\]
\item If $\Delta$ is a connected, orientable $\field$-homology manifold without boundary, then 
\[h''_i(\Delta)=h'_i(\Delta)-\binom{d}{i}\tilde{\beta}_{i-1}(\Delta)=
 h_i(\Delta)-\binom{d}{i}\sum_{j=1}^{i}(-1)^{i-j}\tilde{\beta}_{j-1}(\Delta) \qquad \forall ~ 0\leq i\leq d-1.\]
\item If $\Delta$ is a complex with (at most) one singular vertex $u$, then for all $0\leq i\leq d$,
\[h'_i(\Delta)=h_i(\Delta)-\sum_{j=1}^{i-1}(-1)^{i-j}\left(\binom{d-1}{i-1}\tilde{\beta}_{j-1}(\Delta)+\binom{d-1}{i}\tilde{\beta}_{j-1}(\cost_{\Delta}u)\right).\]
\end{enumerate}
\end{theorem}

\noindent Part 1 of this theorem is due to Stanley \cite{Stanley-75}, part 2 is due to Schenzel \cite{Schenzel-81}, part 3 is \cite[Theorem 1.3]{Novik-Swartz-09:Gorenstein}, and part 4 is a special case of \cite[Theorem 4.7]{Novik-Swartz-12}. When $\Delta$ is a $\field$-homology manifold with boundary, part 4 allows us to compute the $h'$-numbers of $\complet{\Delta}$.  One of the goals of this paper is to understand the $h''$-numbers of $\complet{\Delta}$ where we in addition assume that $\Delta$ is connected and orientable. This requires some results on the local cohomology of $\field[\complet{\Delta}]$ that we review in the next subsection.

\subsection{Local cohomology and Gr\"abe's theorem}
Let $M$ be an arbitrary finitely-generated graded $A$-module. We denote by $H^i_\mideal(M)$ the $i$-th local cohomology of $M$ with respect to $\mideal$. 

For a simplicial complex $\Delta$, Gr\"abe \cite{Grabe-84} gave a description of $H^i_\mideal(\field[\Delta])$ and its $A$-module structure in terms of simplicial cohomology of the links of $\Delta$ and the maps between them. When $\Delta$ is a complex with one singular vertex $u$, this description takes the following simple form. For $F\in\Delta$, consider the $i$-th simplicial cohomology of the pair $(\Delta, \cost_\Delta F)$ with coefficients in $\field$:
\[
H^i_F:=\tilde{H}^i(\Delta, \cost_\Delta F; \field)\cong \tilde{H}^{i-|F|}(\lk_\Delta F; \field). \]
In particular, $H^i_\emptyset=\tilde{H}^i(\Delta, \emptyset;\field)=\tilde{H}^i(\Delta; \field)$. If $G\subseteq F\in\Delta$, we let $\iota^\ast: H^i_F(\Delta)\to H^i_G(\Delta)$ be the map induced by inclusion $\iota: \cost_\Delta G \to \cost_\Delta F$.

\begin{theorem} \label{local-cohom-1-singul}
{\rm[Gr\"abe]} Let $\Delta$ be a $(d-1)$-dimensional simplicial complex with one singular vertex $u$, and let $-1\leq i <d-1$. Then
\[
H^{i+1}_\mideal(\field[\Delta])_{-j}=\left\{\begin{array}{ll}
0 & \mbox{ if $j<0$},\\
H^i_\emptyset(\Delta) & \mbox{ if $j=0$},\\
H^i_{\{u\}}(\Delta) & \mbox{ if $j>0$}.
\end{array}
\right.
\]
For every vertex $w\neq u$ and any integer $j$, the map $\cdot x_w: H^{i+1}_\mideal(\field[\Delta])_{-(j+1)} \to H^{i+1}_\mideal(\field[\Delta])_{-j}$ is the zero map; on the other hand, 
\[\cdot x_u= \left\{\begin{array}{ll}
\mbox{$0$-map} & \mbox{ if $j<0$},\\
\iota^*: H^i_{\{u\}}(\Delta) \to H^i_\emptyset(\Delta) & \mbox{ if $j=0$},\\
\mbox{identity map: $H^i_{\{u\}}(\Delta) \to H^i_{\{u\}}(\Delta)$} &  \mbox{ if $j>0$}.
\end{array}
\right.
\]
\end{theorem}

The description of $H^d_\mideal(\field[\Delta])$ is quite a bit more involved. To this end, for a monomial $\rho\in A$, define the {\em support} of $\rho$ by $s(\rho):=\{v\in V \mid x_v \mbox{ divides }\rho\}$. Let $\M(\Delta)$ be the set of all monomials in $A$ whose support is in $\Delta$, and let $\M_j(\Delta):=\{\rho\in\M(\Delta) \mid \deg(\rho)=j\}$.

\begin{theorem} \label{Grabe-d}
{\rm[Gr\"abe]} Let $\Delta$ be any $(d-1)$-dimensional simplicial complex. Then for $j\in\ZZ$, 
\[ H^d_{\mideal}(\field[\Delta])_{-j}=\bigoplus_{\rho\in\M_j(\Delta)} \HH_{\rho}, \quad \mbox{where} \quad \HH_{\rho}=H^{d-1}_{s(\rho)}(\Delta),\]
and the $A$-structure on the $\rho$-th component of the right-hand side is given by
\[
\cdot x_\ell = \left\{\begin{array}{ll} 
\mbox{$0$-map} & \mbox{if $\ell\notin s(\rho)$},\\
\iota^*: H^{d-1}_{s(\rho)}(\Delta) \to H^{d-1}_{s(\rho/x_\ell)}(\Delta) & \mbox{if $\ell\in s(\rho)$, but $\ell\notin s(\rho/x_\ell)$},\\
\mbox{identity map: $H^{d-1}_{s(\rho)}(\Delta) \to H^{d-1}_{s(\rho/x_\ell)}(\Delta)$} & \mbox{if $\ell\in s(\rho)$, and $\ell\in s(\rho/x_\ell)$}.
\end{array}
\right.
\]
\end{theorem}

\section{The Gorenstein and Weak Lefschetz properties} \label{sec:Gorenstein}
If $\Delta$ is a $\field$-homology sphere and $\Theta$ is an arbitrary l.s.o.p.~for $\field[\Delta]$, then $\field(\Delta,\Theta)$ is Gorenstein (see \cite[Theorem II.5.1]{Stanley-96}). This result was extended in \cite[Theorem 1.4]{Novik-Swartz-09:Gorenstein} to connected, orientable $\field$-homology manifolds without boundary: if $\Delta$ is such a complex and $\Theta$ is an l.s.o.p.~for $\field[\Delta]$, then $\overline{\field(\Delta,\Theta)}$ is Gorenstein. Here we further extend this result to manifolds with boundary. 

Throughout this section, we let $\Delta$ be a $\field$-homology manifold with boundary. We assume that $\Delta$ is $(d-1)$-dimensional and has vertex set $V$, and so $\complet{\Delta}$ has vertex set $V_0:=V\cup\{v_0\}$. The main result of this section is:

\begin{theorem} \label{thm-Gorenstein}
Let $\Delta$ be a connected, orientable $\field$-homology manifold with boundary and let $\Theta$ be a generic l.s.o.p.~for $\field[\complet{\Delta}]$. Then $\overline{\field(\complet{\Delta},\Theta)}$ is Gorenstein.
\end{theorem}

The proof relies on a few lemmas. For these lemmas we fix a vertex $v$ of $\Delta$. (Hence $v$ is a non-singular vertex of $\complet{\Delta}$.)

\begin{lemma} \label{lem:homol-of-complex-and-costar}
Let $\Delta$ be a $(d-1)$-dimensional, $\field$-homology manifold with boundary, and let $v$ be a vertex of $\Delta$. Then for all $j\leq d-3$,
\[
\tilde{\beta}_j(\complet{\Delta})=\tilde{\beta}_j(\cost_{\complet{\Delta}}v) \quad \mbox{and} \quad \tilde{\beta}_j(\Delta)=\tilde{\beta}_j(\cost_{\Delta}v).
\]
\end{lemma}
\begin{proof}
The proof is a simple application of the Mayer-Vietoris argument. Indeed, since $v\neq v_0$, the link $\complet{L}:=\lk_{\complet{\Delta}}v$ is a $\field$-homology $(d-2)$-sphere, while the link $L:=\lk_\Delta v$ is a $\field$-homology $(d-2)$-sphere or a $(d-2)$-ball. In either case, $\tilde{\beta}_j(\complet{L})=\tilde{\beta}_j(L)=0$ for all $j\leq d-3$. Since, the stars $\st_\Delta v$ and $\st_{\complet{\Delta}}v$ are acyclic, considering the following portion 
\[
\tilde{H}_j(L)\to\tilde{H}_j(\cost_\Delta v)\oplus\tilde{H}_j(\st_\Delta v)\to\tilde{H}_j(\Delta)\to\tilde{H}_{j-1}(L)
\]
of the Mayer-Vietoris sequence for $\Delta$ and its analog for $\complet{\Delta}$ yields the result.
\end{proof}

 The following lemma is a generalization of \cite[Proposition 4.24]{Swartz-09}. We set $A:=\field[x_u \mid u\in V_0]$ and $A^v:=\field[x_u \mid u\in V_0\setminus\{v\}]$.
Observe that $\field[\complet{\Delta}]$ and $\field[\cost_{\complet{\Delta}}v]$ have natural $A$-module structures (where multiplication by $x_v$ on $\field[\cost_{\complet{\Delta}}v]$  is the zero map), while $\field[\lk_{\complet{\Delta}}v]$ has a natural $A^v$-module structure (if $u\neq v$ is not in the link of $v$, then multiplication by $x_u$ is the zero map). Let $\Theta=\theta_1,\ldots,\theta_d\in A$ be a generic l.s.o.p.~for  $\field[\complet{\Delta}]$, and hence also for $\field[\cost_{\complet{\Delta}}v]$. Since $\Theta$ is generic, $\theta_1$ has non-vanishing coefficients. So by scaling the variables if necessary, we can work in an isomorphic setting and assume w.l.o.g.~that {\em all} coefficients of $\theta_1$ are equal to $1$.
Let $\theta'_1:=\theta_1-x_v$, and for $j>1$, let $\theta'_j=\theta_j-c_j\theta_1$ where $c_j$ is the coefficient of $x_v$ in $\theta_j$. Then $\theta'_1, \theta'_2,\ldots,\theta'_d$ can be viewed as elements of $A^v$, with $\Theta^v=\{\theta'_2,\ldots,\theta'_d\}\subset A^v$ forming an l.s.o.p.~for $\field[\lk_{\complet{\Delta}}v]$. Furthermore, the ring $\field(\lk_{\complet{\Delta}}v, \Theta^v)$  inherits an $A^v$-module structure, and defining
\[
x_v\cdot y:=-\theta'_1\cdot y \quad \mbox{for }y\in \field(\lk_{\complet{\Delta}}v, \Theta^v)
\]
extends it to an $A$-module structure.

\begin{lemma} \label{analog-of-Prop4.24}
 Let $\Delta$ be a $(d-1)$-dimensional, connected, orientable $\field$-homology manifold with boundary, and let $v$ be a vertex of $\Delta$. Then 
the map
\[\phi_v: \field(\lk_{\complet{\Delta}}v, \Theta^v) \to (x_v)\field(\complet{\Delta},\Theta) \quad \mbox{given by } z \to x_v\cdot z,
\]
is well-defined and is a graded isomorphism of $A$-modules (of degree $1$).
\end{lemma}
\begin{proof} The proof of \cite[Proposition 4.24]{Swartz-09} shows that $\phi_v$ is a well-defined and surjective homomorphism of $A$-modules. Thus to complete the proof, it suffices to check that for $1\leq i\leq d$, the dimensions of $\field$-spaces $\big(\field(\lk_{\complet{\Delta}}v, \Theta^v)\big)_{i-1}$ and $\big((x_v)\field(\complet{\Delta},\Theta)\big)_i$ agree.
Since $v\neq v_0$, the link $\lk_{\complet{\Delta}}v$ is a $\field$-homology sphere, and so 
\begin{equation} \label{dim-lk}
\dim_\field \big(\field(\lk_{\complet{\Delta}}v, \Theta^v)\big)_{i-1}=h_{i-1}(\lk_{\complet{\Delta}}v) \quad \mbox{for all }i\leq d.
\end{equation}

To compute $\dim_\field \big((x_v)\field(\complet{\Delta},\Theta)\big)_i$ for $i\leq d$, consider the following exact sequence, induced by the natural surjection $\field[\complet{\Delta}]\to \field[\cost_{\complet{\Delta}}v]$, 
\begin{equation} \label{exact-x_v}
0\to (x_v)\field(\complet{\Delta},\Theta) \to \field(\complet{\Delta},\Theta) \to \field(\cost_{\complet{\Delta}}v, \Theta)\to 0.
\end{equation}
If $i=d$, then, since $\Delta$ is connected and orientable, $\dim_\field \big(\field(\complet{\Delta},\Theta)\big)_d=\tilde{\beta}_{d-1}(\complet{\Delta})=1$ while $\dim_\field \field(\cost_{\complet{\Delta}}v, \Theta)_d=\tilde{\beta}_{d-1}(\cost_{\complet{\Delta}}v)=0$. Hence, in this case eq.~\eqref{exact-x_v} implies that
\[
\dim_\field \big((x_v)\field(\complet{\Delta},\Theta)\big)_d=\dim_\field \big(\field(\complet{\Delta},\Theta)\big)_d=1
=\tilde{\beta}_{d-1}(\lk_{\complet{\Delta}}v)=\dim_\field \big(\field(\lk_{\complet{\Delta}}v, \Theta^v)\big)_{d-1},
\]
as desired. 

Thus for the rest of the proof we assume that $1\leq i\leq d-1$. Since both $\complet{\Delta}$ and $\cost_{\complet{\Delta}}v$ are complexes with at most one singular vertex, namely $v_0$, and since $\cost_{\complet{\Delta}}v_0=\Delta$, we infer from Theorem \ref{known-h'-h''}(4) that
\begin{equation} \label{h'}
\dim_\field \big(\field(\complet{\Delta},\Theta)\big)_i-h_i(\complet{\Delta}) =
-\sum_{j=1}^{i-1}(-1)^{i-j}\left(\binom{d-1}{i-1}\tilde{\beta}_{j-1}(\complet{\Delta})+\binom{d-1}{i}\tilde{\beta}_{j-1}(\Delta)\right),
\end{equation}
and that a similar expression holds for $\dim_\field\big(\field(\cost_{\complet{\Delta}}v,\Theta)\big)_i-h_i(\cost_{\complet{\Delta}}v)$: to obtain it,  simply replace all occurrences of $\complet{\Delta}$ on the right-hand side of \eqref{h'} with $\cost_{\complet{\Delta}}v$ and those of $\Delta$ with $\cost_{\Delta}v$. Since according to Lemma \ref{lem:homol-of-complex-and-costar}, for $i\leq d-1$, these replacements do not affect the value of the right-hand side of \eqref{h'}, we conclude that for all $1\leq i\leq d-1$,
\begin{eqnarray*}
\dim_\field \big((x_v)\field(\complet{\Delta},\Theta)\big)_i & \stackrel{\mbox{\tiny{by \eqref{exact-x_v}}}}{=}& 
\dim_\field\big(\field(\complet{\Delta},\Theta)\big)_i- \dim_\field\big(\field(\cost_{\complet{\Delta}}v,\Theta)\big)_i 
= h_i(\complet{\Delta})-h_i(\cost_{\complet{\Delta}}v) \\
&=& h_{i-1}(\lk_{\complet{\Delta}}v) \stackrel{\mbox{\tiny{by \eqref{dim-lk}}}}{=}
\dim_\field \big(\field(\lk_{\complet{\Delta}}v, \Theta^v)\big)_{i-1}, 
\end{eqnarray*}
where the penultimate step uses \cite[Lemma 4.1]{Athanasiadis-11}. The result follows.
\end{proof}

We are now in a position to prove Theorem \ref{thm-Gorenstein}. Our proof follows the same outline as the proof of \cite[Theorem 1.4]{Novik-Swartz-09:Gorenstein} with an additional twist at the end.

\smallskip\noindent {\it Proof of Theorem \ref{thm-Gorenstein}: \ }
Let $\Delta$ be a $(d-1)$-dimensional, connected, orientable $\field$-homology manifold with boundary, and let $\Theta$ be a generic l.s.o.p.~for $\field[\complet{\Delta}]$. (As before, we assume w.l.o.g.~that all coefficients of $\theta_1$ are equal to $1$.) Then $\dim_\field \overline{\field(\complet{\Delta},\Theta)}_d =\tilde{\beta}_{d-1}(\complet{\Delta})=1$. Hence, we only need to verify that the socle of $\overline{\field(\complet{\Delta},\Theta)}=\field(\complet{\Delta},\Theta)/\Soc^\circ \field(\complet{\Delta},\Theta)$ vanishes in all degrees $j\neq d$. Since $\big(\Soc^\circ \field(\complet{\Delta},\Theta)\big)_d=0$ and $\big(\Soc^\circ \field(\complet{\Delta},\Theta)\big)_{d-1}=\big(\Soc \field(\complet{\Delta},\Theta)\big)_{d-1}$, this does hold for $j=d-1$.

Now, let $j\leq d-2$, and let $y\in \field(\complet{\Delta},\Theta)_j$ be such that $x_v\cdot y\in (\Soc \field(\complet{\Delta},\Theta))_{j+1}$ for every vertex $v$ of $\complet{\Delta}$. We must show that $y\in \Soc \field(\complet{\Delta},\Theta)$. Assume first that $v\neq v_0$. Then the isomorphism of Lemma \ref{analog-of-Prop4.24} implies that $y^v:=\phi_v^{-1}(x_v\cdot y)\in (\field(\lk_{\complet{\Delta}}v, \Theta^v))_j$ is in the socle of $\field(\lk_{\complet{\Delta}}v, \Theta^v)$. Since $\lk_{\complet{\Delta}}v$ is a $\field$-homology $(d-2)$-sphere,  $\field(\lk_{\complet{\Delta}}v, \Theta^v)$ is Gorenstein, and hence its socle vanishes in all degrees $\leq d-2$. Therefore, $y^v=0$. We conclude that
\begin{equation} \label{v-neq-v0}
x_v\cdot y=\phi_v(y^v)=0 \mbox{ in } \field(\complet{\Delta},\Theta) \quad \mbox{for all } v\neq v_0.
\end{equation}

Finally, to show that $x_{v_0}\cdot y=0$ in $\field(\complet{\Delta},\Theta)$, recall that $\theta_1=x_{v_0}+\sum_{v\neq v_0}x_v$, and so
\begin{equation} \label{v0}
\theta_1\cdot y=x_{v_0}\cdot y + \sum_{v\neq v_0} x_v\cdot y.
\end{equation}
The left-hand side of \eqref{v0} is zero in $\field(\complet{\Delta},\Theta)=\field[\complet{\Delta}]/\Theta\field[\complet{\Delta}]$. Furthermore, by \eqref{v-neq-v0},  all summands on the right-hand side of \eqref{v0}, except possibly $x_{v_0}\cdot y$, are zeros in $\field(\complet{\Delta},\Theta)$. Thus $x_{v_0}\cdot y$ must be zero in $\field(\complet{\Delta},\Theta)$. The result follows. 
\endproof

We now turn to some consequences of Theorem \ref{thm-Gorenstein}. As the Hilbert function of a Gorenstein graded $\field$-algebra of Krull dimension zero is always symmetric, one immediate corollary is
\begin{corollary} \label{symmetry}
Let $\Delta$ be a $(d-1)$-dimensional, connected, orientable $\field$-homology manifold with boundary. Then 
$h''_i(\complet{\Delta})=h''_{d-i}(\complet{\Delta})$ for all $0\leq i \leq d$.
\end{corollary}

Let $\Gamma$ be a $\field$-homology $(m-1)$-sphere  or $(m-1)$-ball. We say that $\Gamma$ has the {\em weak Lefschetz property over $\field$} (the WLP, for short) if for a generic lsop $\Theta$ for $\field[\Gamma]$ and an additional generic linear form $\omega$, the map $\cdot \omega: \field(\Gamma,\Theta)_{\lfloor \frac{m}{2}\rfloor} \to \field(\Gamma,\Theta)_{\lfloor \frac{m}{2}\rfloor+1}$ is surjective. It was proved in \cite[Theorem 3.2]{Novik-Swartz-09:Gorenstein} that if $\Lambda$ is a $(d-1)$-dimensional, connected, orientable $\field$-homology manifold without boundary, and if all but at most $d$ vertex links of $\Lambda$ have the WLP over $\field$, then for generic $\Theta$ and $\omega$, the map $\cdot \omega: \overline{\field(\Lambda,\Theta)}_i \to \overline{\field(\Lambda,\Theta)}_{i+1}$ is an injection for $i<\lfloor \frac{d}{2}\rfloor$ and is a surjection for  $i\geq \lceil \frac{d}{2}\rceil$. The proof relied on \cite[Theorem 4.26]{Swartz-09} and on the Gorenstein property of $\overline{\field(\Lambda,\Theta)}$ established in \cite[Theorem 1.4]{Novik-Swartz-09:Gorenstein}). Noting that \cite[Theorem 4.26]{Swartz-09} continues to hold for $\complet{\Delta}$ and using Theorem~\ref{thm-Gorenstein} instead of \cite[Theorem 1.4]{Novik-Swartz-09:Gorenstein}, but leaving the rest of the proof of \cite[Theorem 3.2]{Novik-Swartz-09:Gorenstein} intact, yields the following generalization:

\begin{corollary} \label{WLP}
Let $\Delta$ be a $(d-1)$-dimensional connected, orientable $\field$-homology manifold with boundary. 
\begin{enumerate}
\item If $d\geq 4$, then the map $\cdot \omega: \overline{\field(\complet{\Delta},\Theta)}_i \to \overline{\field(\complet{\Delta},\Theta)}_{i+1}$
is an injection for $i\leq 1$ and is a surjection for $i\geq d-2$.
\item If for all vertices $v$ of $\Delta$, the link $\lk_{\complet{\Delta}}v$ has the WLP over $\field$, then the map 
$\cdot \omega: \overline{\field(\complet{\Delta},\Theta)}_i \to \overline{\field(\complet{\Delta},\Theta)}_{i+1}$ is an injection for all $i<\lfloor \frac{d}{2}\rfloor$ and is a surjection for all $i\geq \lceil \frac{d}{2}\rceil$.
\end{enumerate}
\end{corollary}

\begin{remark} \label{rem:WLP}
By a result of Stanley \cite{Stanley-80}, the boundary complexes of all simplicial polytopes have the WLP over $\Q$. Furthermore, it follows from \cite[Cor.~3.5]{Murai-10:Shifting} and \cite{Whiteley-90} that all triangulations of $2$-dimensional spheres have the WLP over {\em any} infinite field. (This result is the reason no assumption on vertex links is needed in part 1 of the corollary.) A very recent preprint by Adiprasito \cite{Adiprasito-g} announces a spectacular generalization of these theorems: for an arbitrary infinite field $\field$, every $\field$-homology sphere has the weak Lefschetz (and even strong Lefschetz) property over $\field$, and so the hypothesis of the WLP assumption in the statement of Corollary \ref{WLP} as well as in the rest of the paper might be unnecessary.
\end{remark}

To apply results of this section to the study of face numbers of homology manifolds with boundary, we first need to work out the $h''$-numbers of $\complet{\Delta}$, that is, the Hilbert function of $\overline{\field(\complet{\Delta},\Theta)}$. This is done in the next section.

\section{The $h''$-numbers of $\complet{\Delta}$} \label{sec:h''-and-socles}
In this section we prove the following extension of Theorem \ref{known-h'-h''}(3) to manifolds with boundary. 

\begin{theorem} \label{soc-and-h''}
Let $\Delta$ be a $(d-1)$-dimensional, connected, orientable $\field$-homology manifold with boundary and let $\Theta$ be a generic l.s.o.p.~for $\field[\complet{\Delta}]$. Then  for all $i<d$,
\begin{enumerate}
\item $\dim_\field \big(\Soc \field(\complet{\Delta},\Theta)\big)_i=\binom{d-1}{i-1}\tilde{\beta}_{i-1}(\complet{\Delta})+\binom{d-1}{i}\tilde{\beta}_{i-1}(\Delta)$, and
\item $h''_i(\complet{\Delta})=h_i(\complet{\Delta})-\sum_{j=1}^{i}(-1)^{i-j}\left(\binom{d-1}{i-1}\tilde{\beta}_{j-1}(\complet{\Delta})+\binom{d-1}{i}\tilde{\beta}_{j-1}(\Delta)\right)$.
\end{enumerate}
\end{theorem}

\begin{remark} \label{rem:h''-v2}
It is instructive to rewrite both formulas of the theorem purely in terms of $\Delta$. Indeed, by connectivity, $\tilde{\beta}_0(\complet{\Delta})=\tilde{\beta}_0(\Delta)=0$, while
\[\tilde{H}_{j-1}(\complet{\Delta};\field)\cong \tilde{H}_{j-1}(\complet{\Delta}, \st_{\complet{\Delta}}v_0;\field)\cong \tilde{H}_{j-1}(\Delta, \partial\Delta;\field)\cong \tilde{H}_{d-j}(\Delta;\field) \quad  \forall ~ j< d,\]
where the first step follows from the acyclicity of vertex stars, the second by excision, and the third by Poincar\'e-Lefschetz duality. Furthermore,  $h_i(\complet{\Delta})=h_i(\Delta)+h_{i-1}\big(\lk_{\complet{\Delta}}v_0\big)=h_i(\Delta)+h_{i-1}(\partial\Delta)$ (see \cite[Lemma 4.1]{Athanasiadis-11}).
Thus, for $i<d$, Theorem \ref{soc-and-h''} can be rewritten as
\begin{enumerate}
\item $\dim_\field \big(\Soc \field(\complet{\Delta},\Theta)\big)_i=\binom{d-1}{i}\tilde{\beta}_{i-1}(\Delta)+\binom{d-1}{i-1}\tilde{\beta}_{d-i}(\Delta)$;
\item
$h''_i(\complet{\Delta})=h_i(\Delta)+h_{i-1}(\partial\Delta)-\sum_{j=2}^{i}(-1)^{i-j}\left(\binom{d-1}{i}\tilde{\beta}_{j-1}(\Delta)+\binom{d-1}{i-1}\tilde{\beta}_{d-j}(\Delta)\right).$
\end{enumerate}
Note that if $\Delta$ is a connected, orientable $\field$-homology manifold without boundary, then (1) $\complet{\Delta}=\Delta$, (2)  $\tilde{\beta}_{d-j}(\Delta)=\tilde{\beta}_{j-1}(\Delta)$ for all $1<j<d$ (by Poincar\'e duality), and (3) $h_{i-1}(\partial\Delta)=0$ for all $i$ (since $\partial\Delta=\emptyset$). In this case, the above formula for $h''_i(\complet{\Delta})$ reduces to Theorem \ref{known-h'-h''}(3).
\end{remark}

To prove Theorem \ref{soc-and-h''}, several lemmas are in order. As in the previous section,  we continue to assume that $\Theta$ is a generic l.s.o.p.~for $\field[\complet{\Delta}]$ and that all coefficients of $\theta_1$ are equal to 1.

\begin{lemma} \label{mod-theta1-decompos}
Let $\Delta$ be a $(d-1)$-dimensional $\field$-homology manifold with boundary. Then 
\[
\big(\Soc\field(\complet{\Delta},\Theta)\big)_i\cong \left(\bigoplus_{j=0}^{d-2}\binom{d-1}{j}\big(H^j_\mideal\big(\field[\complet{\Delta}]/\theta_1\field[\complet{\Delta}]\big)\big)_{i-j}\right)\bigoplus (\SB)_{i-(d-1)} \quad \forall i\in\Z,
\]
where $\SB$ is a graded submodule of $\Soc H^{d-1}_\mideal\big(\field[\complet{\Delta}]/\theta_1\field[\complet{\Delta}]\big)$.
Furthermore, for $j\leq d-2$,
\[
\dim_\field \big(H^j_\mideal\big(\field[\complet{\Delta}]/\theta_1\field[\complet{\Delta}]\big)\big)_\ell=
\left\{\begin{array}{ll}
\tilde{\beta}_j(\complet{\Delta}) & \mbox{ if $\ell=1$}\\
\tilde{\beta}_{j-1}(\Delta) &  \mbox{ if $\ell=0$}\\
0 & \mbox{ otherwise}.
\end{array}
\right.
\]
\end{lemma}
\proof Since $\complet{\Delta}$ has at most one singularity, Lemma 4.3(2) of \cite{Novik-Swartz-12} implies that $\field[\complet{\Delta}]/\theta_1\field[\complet{\Delta}]$ is a Buchsbaum $A$-module of Krull dimension $d-1$. The first part of the statement then follows from \cite[Theorem 2.2]{Novik-Swartz-09:Socles}, while the second part follows from  \cite[Lemma 4.3(1) and Theorem 4.7]{Novik-Swartz-12}.
\endproof

We now turn our attention to the submodule $\SB$ of $\Soc H^{d-1}_\mideal\big(\field[\complet{\Delta}]/\theta_1\field[\complet{\Delta}]\big)$.
\begin{proposition} \label{SB}
Let $\Delta$ be a $(d-1)$-dimensional, connected, orientable $\field$-homology manifold with boundary. Then, for all $\ell\leq -1$, $\big(\Soc H^{d-1}_\mideal\big(\field[\complet{\Delta}]/\theta_1\field[\complet{\Delta}]\big)\big)_\ell=0$, and hence $(\SB)_\ell=0$.
\end{proposition}
\proof
Since $\depth \field[\complet{\Delta}]\geq 1$, $\theta_1$ is a non-zero divisor on $\field[\complet{\Delta}]$; in other words, the sequence
\[
0\to \field[\complet{\Delta}](-1) \stackrel{\cdot \theta_1}{\longrightarrow} \field[\complet{\Delta}] \longrightarrow \field[\complet{\Delta}]/\theta_1\field[\complet{\Delta}]\to 0
\]
is exact. (For a graded $A$-module $M$, $M(-1)$ denotes $M$ with grading defined by $M(-1)_\ell=M_{\ell-1}$.) 

The above sequence induces a long exact sequence in local cohomology. In particular, the part
\[
H^{d-1}_\mideal\big(\field[\complet{\Delta}]\big)(-1) \stackrel{\cdot \theta_1}{\longrightarrow} H^{d-1}_\mideal\big(\field[\complet{\Delta}]\big) \longrightarrow H^{d-1}_\mideal\big(\field[\complet{\Delta}]/\theta_1\field[\complet{\Delta}]\big)\longrightarrow H^{d}_\mideal\big(\field[\complet{\Delta}]\big)(-1)  \stackrel{\cdot \theta_1}{\longrightarrow} H^{d}_\mideal\big(\field[\complet{\Delta}]\big)
\]
is exact.  Thus, $H^{d-1}_\mideal\big(\field[\complet{\Delta}]/\theta_1\field[\complet{\Delta}]\big)$, considered as a vector space, is isomorphic to the direct sum of $\C:=\Coker \big[H^{d-1}_\mideal\big(\field[\complet{\Delta}]\big)(-1) \stackrel{\cdot \theta_1}{\longrightarrow} H^{d-1}_\mideal\big(\field[\complet{\Delta}]\big)\big]$ and $\K:=\Ker \big[H^{d}_\mideal\big(\field[\complet{\Delta}]\big)(-1)  \stackrel{\cdot \theta_1}{\longrightarrow} H^{d}_\mideal\big(\field[\complet{\Delta}]\big)\big]$. Futhermore, on the $\K$-part of $H^{d-1}_\mideal\big(\field[\complet{\Delta}]/\theta_1\field[\complet{\Delta}]\big)$, the $A$-module structure is induced by the $A$-module structure on $H^{d}_\mideal\big(\field[\complet{\Delta}]\big)$.

Since $\complet{\Delta}$ has (at most) one singular vertex, namely $v_0$,   Theorem \ref{local-cohom-1-singul} implies that for $\ell\leq -1$, the map $\cdot\theta_1: \big(H^{d-1}_\mideal\big(\field[\complet{\Delta}]\big)\big)_{\ell-1}\to  \big(H^{d-1}_\mideal\big(\field[\complet{\Delta}]\big)\big)_{\ell}$ is the identity map. Hence its cokernel, $\C_\ell$, vanishes for all $\ell\leq -1$. Therefore, it only remains to show that the socle $(\Soc \K)_\ell$, vanishes for all $\ell\leq -1$.  Indeed, by definition of socles,
\begin{eqnarray*}
(\Soc \K)_\ell &=&\left(\Soc \Ker \big[\cdot \theta_1: H^{d}_\mideal\big(\field[\complet{\Delta}]\big)(-1)  \longrightarrow H^{d}_\mideal\big(\field[\complet{\Delta}]\big)\big]\right)_\ell \\
&=& \left(\Soc H^{d}_\mideal\big(\field[\complet{\Delta}]\big)(-1)\right)_\ell=\left(\Soc H^{d}_\mideal\big(\field[\complet{\Delta}]\big)\right)_{\ell-1}.
\end{eqnarray*}
The following lemma verifies that the latter term vanishes, and thus completes the proof. 
\endproof

\begin{lemma} \label{vanishing-socle}
Let $\Delta$ be a $(d-1)$-dimensional, connected, orientable $\field$-homology manifold with boundary. Then, for all $\ell\geq 2$,
$\left(\Soc H^{d}_\mideal\big(\field[\complet{\Delta}]\big)\right)_{-\ell}=0.$
\end{lemma}
\proof Recall that by Theorem \ref{Grabe-d},
\begin{equation} \label{eq:H_d}
H^d_{\mideal}(\field[\complet{\Delta}])_{-\ell}=\bigoplus_{\rho\in\M_\ell(\complet{\Delta})} \HH_{\rho}, \quad \mbox{where} \quad \HH_{\rho}=H^{d-1}_{s(\rho)}(\complet{\Delta}).
\end{equation}
Fix $\ell\geq 2$, and let $\rho\in\M_\ell(\complet{\Delta})$. Then either $\rho$ is divisible by $x_v^2$ for some vertex $v$ of $\complet{\Delta}$ (possibly $v_0$) or $\rho$ is a squarefree monomial whose support has size at least two: $s(\rho)\supseteq\{v, w\}$. In the former case, by Theorem \ref{Grabe-d}, the multiplication map $\cdot x_v: \HH_{\rho} \to \HH_{\rho/x_v}$ is the identity map, and so no non-zero element of $\HH_{\rho}$ is in the socle. In the latter case, at least one of $v,w$ is not $v_0$. Assume without loss of generality that $w\neq v_0$, and consider the map $\cdot x_v: \HH_{\rho} \to \HH_{\rho/x_v}$, which by Theorem \ref{Grabe-d} is simply $\iota^\ast: H^{d-1}_{s(\rho)}(\complet{\Delta}) \to H^{d-1}_{s(\rho/x_v)}(\complet{\Delta})$. We will show that this map is an isomorphism, and hence that no non-zero element of $\HH_{\rho}$ is in the socle in this case as well.

Our argument is similar to the one used in the proof of \cite[Theorem 2.1]{Novik-Swartz-09:Gorenstein}. Denote by $\|\complet{\Delta}\|$ the geometric realization of $\complet{\Delta}$, and by $b(\rho)$ and $b(\rho/x_v)$ the barycenters of realizations of faces $s(\rho)$ and $s(\rho/x_v)$, respectively. Consider the following commutative diagram, where the maps $f^\ast$ and $j^\ast$ are induced by inclusion:
$$
\begin{CD}
\tilde{H}^{d-1}\big(\|\complet{\Delta}\|\big) @>(j^\ast)^{-1}>> 
\tilde{H}^{d-1}\big(\|\complet{\Delta}\|, \|\complet{\Delta}\| - b(\rho/x_v)\big)
 @>f^\ast>>  \tilde{H}^{d-1}\big(\complet{\Delta}, \cost_{\complet{\Delta}}s(\rho/x_v)\big) \\
 @|  @. @A\iota^\ast AA \\
 \tilde{H}^{d-1}\big(\|\complet{\Delta}\|\big) @>(j^\ast)^{-1}>> 
\tilde{H}^{d-1}\big(\|\complet{\Delta}\|, \|\complet{\Delta}\| - b(\rho)\big) 
@>f^\ast>>  \tilde{H}^{d-1}\big(\complet{\Delta}, \cost_{\complet{\Delta}}s(\rho)\big). 
\end{CD}
$$
The two maps $f^\ast$ are isomorphisms by the usual deformation retractions. Since $w\neq v_0$ and $w\in s(\rho/x_v) \subset s(\rho)$, the links $\lk_{\complet{\Delta}} s(\rho)$ and $\lk_{\complet{\Delta}} s(\rho/x_v)$ are $\field$-homology spheres, so the four $\field$-spaces on the right and in the middle of the diagram are $1$-dimensional. Furthermore, since $\Delta$ is connected and orientable, the $\field$-spaces on the left of the diagram are $1$-dimensional and the two $j^\ast$-maps are isomorphisms, so that $(j^\ast)^{-1}$-maps are well-defined and are isomorphisms as well. This implies that $\iota^\ast$ is an isomorphism and completes the proof.
\endproof

We are now ready to prove Theorem \ref{soc-and-h''}.

\smallskip\noindent {\it Proof of Theorem \ref{soc-and-h''}: \ } 
We prove both parts simulataneously. If $d=2$, then $\complet{\Delta}$ is a circle, in which case the statement is known. So assume $d\geq 3$. Lemma \ref{mod-theta1-decompos} and Proposition \ref{SB} imply that the formula for the dimension of the socle holds for all $i\leq d-2$. Together with Theorem \ref{known-h'-h''}(4) and  Definition \ref{def-h'-h''}, this also implies that the formula for $h''_i(\complet{\Delta})$ holds for all $i\leq d-2$. Thus, it only remains to show that the theorem holds for $i=d-1$. Since by Corollary \ref{symmetry}, the $h''$-numbers of $\complet{\Delta}$ are symmetric, to complete the proof of both parts, it suffices to check that  the proposed expression for $h''_{d-1}(\complet{\Delta})$ is equal to $h''_1(\complet{\Delta})=h_1(\complet{\Delta})$.

Let $\tilde{\chi}$ denote the reduced Euler characteristic. Note that since $\tilde{\beta}_{d-1}(\complet{\Delta})=1$ and $\tilde{\beta}_{d-1}(\Delta)=0$, the proposed expression for $h''_{d-1}(\complet{\Delta})$, $h_{d-1}(\complet{\Delta})-\sum_{j=1}^{d-1}(-1)^{d-j-1}\left[(d-1)\tilde{\beta}_{j-1}(\complet{\Delta})+\tilde{\beta}_{j-1}(\Delta)\right]$, can be rewritten as
\[
h_{d-1}(\complet{\Delta})- (d-1)\left(1+(-1)^d\tilde{\chi}(\complet{\Delta})\right)-(-1)^d\tilde{\chi}(\Delta).
\]
Thus to complete the proof, we only need to verify that 
\[
h_{d-1}(\complet{\Delta})=h_{1}(\complet{\Delta})+(d-1)\left(1+(-1)^d\tilde{\chi}(\complet{\Delta})\right)+(-1)^d\tilde{\chi}(\Delta).
\]

To do so, observe that for all $i$,
\[
f_i(\complet{\Delta})=f_i(\Delta)+f_i(\st_{\complet{\Delta}} v_0)- f_i(\partial\Delta).
\]
This, together with the fact that vertex stars are contractible, implies that
\begin{equation} \label{eq:chi}
\tilde{\chi}(\partial\Delta)=\tilde{\chi}(\Delta)+\tilde{\chi}(\st_{\complet{\Delta}} v_0)-\tilde{\chi}(\complet{\Delta})=
\tilde{\chi}(\Delta)-\tilde{\chi}(\complet{\Delta}).
\end{equation}
Finally, according to \cite[Theorem 3.1]{Novik-Swartz-12},
\begin{eqnarray*}
 h_{d-1}(\complet{\Delta})&=&h_{1}(\complet{\Delta})+d\left(1+(-1)^d\tilde{\chi}(\complet{\Delta})\right)-\left(1+(-1)^{d-1}\tilde{\chi}(\partial\Delta)\right)\\
&\stackrel{\mbox{\tiny{by \eqref{eq:chi}}}}{=}& 
h_{1}(\complet{\Delta})+(d-1)+  d(-1)^d \tilde{\chi}(\complet{\Delta})+(-1)^d\left(\tilde{\chi}(\Delta)-\tilde{\chi}(\complet{\Delta})\right)\\
&=&h_{1}(\complet{\Delta})+(d-1)\left(1+(-1)^d\tilde{\chi}(\complet{\Delta})\right)+(-1)^d\tilde{\chi}(\Delta).
\end{eqnarray*}
The result follows.
\endproof

\section{Applications: $g$-theorems for manifolds with boundary}
Algebraic results obtained in the two previous sections along with Macaulay's characterization of Hilbert functions of homogeneous quotients of polynomial rings allow us to easily derive several new enumerative results on face numbers of $\field$-homology manifolds with boundary. This section is devoted to results that generalize and are similar in spirit to the $g$-theorem for simplicial polytopes. We follow the custom and define $g_i:=h_i-h_{i-1}$, $g'_i:=h'_i-h'_{i-1}$, and $g''_i:=h''_i-h''_{i-1}$.

We start by recalling that given positive integers $a$ and $i$, there is a unique way to write
\[
a=\binom{a_i}{i}+\binom{a_{i-1}}{i-1}+\cdots+\binom{a_j}{j}, \quad \mbox{where } a_i>a_{i-1}>\cdots>a_j\geq j\geq 1.
\]
Define
\[a^{\langle i\rangle}:=\binom{a_i+1}{i+1}+\binom{a_{i-1}+1}{i}+\cdots+\binom{a_j+1}{j+1} \quad\mbox{and} \quad 0^{\langle i\rangle}:=0.
\]
Macaulay's theorem \cite[Theorem II.2.2]{Stanley-96} asserts that a (possibly infinite) sequence $(b_0, b_1, \ldots)$ of  integers is the Hilbert function of a homogeneous quotient of a polynomial ring if and only if $b_0=1$ and $0\leq b_{\ell+1}\leq b_\ell^{\langle \ell\rangle}$ for all $\ell\geq 1$. A sequence that satisfies these conditions is called an {\em $M$-vector}.

Our first $g$-type result is an extension of \cite[Theorem 3.2]{Novik-Swartz-09:Gorenstein} to manifolds with boundary:
\begin{theorem} \label{g-theorem-v1}
Let $\Delta$ be a $(d-1)$-dimensional, connected, orientable $\field$-homology manifold with boundary, and let $h''_i(\complet{\Delta})=h_i(\Delta)+h_{i-1}(\partial\Delta)-\sum_{j=2}^{i}(-1)^{i-j}\left(\binom{d-1}{i}\tilde{\beta}_{j-1}(\Delta)+\binom{d-1}{i-1}\tilde{\beta}_{d-j}(\Delta)\right)$ for $i<d$ and $h''_d(\complet{\Delta})=1$. Then
\begin{enumerate}
\item $h''_i(\complet{\Delta})=h''_{d-i}(\complet{\Delta})$ for all $0\leq i \leq d$. 
\item If $d\geq 4$, then $\big(1, g''_1(\complet{\Delta}), g''_2(\complet{\Delta})\big)$ is an $M$-vector.
\item If for all vertices $v$ of $\Delta$, $\lk_{\complet{\Delta}}v$ has the WLP over $\field$, then 
$\big(1, g''_1(\complet{\Delta}), g''_2(\complet{\Delta}), \cdots, g''_{\halffloor}(\complet{\Delta})\big)$ is an $M$-vector. 
\end{enumerate}
\end{theorem}
\proof
The expressions for $h''_i(\complet{\Delta})$ are from Remark \ref{rem:h''-v2}(2).  Part 1 is the content of Corollary~\ref{symmetry}. Furthermore, it follows from Corollary \ref{WLP} and Theorem \ref{soc-and-h''}/Remark \ref{rem:h''-v2}  that under our assumptions, for generic $\Theta$ and $\omega$, and for $i\leq 2$ in part 2 and  $i\leq \halffloor$ in part 3,  
$$\dim_\field \left(\overline{\field(\complet{\Delta},\Theta)}/\omega \overline{\field(\complet{\Delta},\Theta)}\right)_i=g''_i(\complet{\Delta},\Theta).$$
Together with Macaulay's theorem, this completes the proof.
\endproof

\begin{remark} Applying the same reasoning to $\field(\complet{\Delta},\Theta)/\bigoplus_{j=0}^{\ell} \big(\Soc \field(\complet{\Delta},\Theta)\big)_j$ instead of $\overline{\field(\complet{\Delta},\Theta)}$, part 3 of Theorem \ref{g-theorem-v1} can be strengthened to the statement that $$\left(1, g''_1(\complet{\Delta}),  \cdots, g''_\ell(\complet{\Delta}), g''_{\ell+1}(\complet{\Delta})+\binom{d-1}{\ell+1}\tilde{\beta}_\ell(\Delta)+\binom{d-1}{\ell}\tilde{\beta}_{d-\ell-1}(\Delta)\right)$$ is an $M$-vector for every $\ell<\halffloor$ (cf.~discussion at the bottom of page 995 in \cite{Novik-Swartz-09:Gorenstein}).
\end{remark}

Our second $g$-type result is an extension of \cite[Theorem 5.4(i)]{Murai-Nevo-14} to manifolds with boundary. To this end, in the spirit of \cite[Section 5]{Murai-Nevo-14}, for a $(d-1)$-dimensional, connected, orientable, $\field$-homology manifold with boundary, $\Delta$, and for $r\leq \lfloor d/2\rfloor$, define 
\begin{eqnarray} 
\tilde{g}_r(\complet{\Delta})&:=&g''_r(\complet{\Delta})-\left(\binom{d-1}{r-1}\tilde{\beta}_{r-1}(\Delta)+\binom{d-1}{r-2}\tilde{\beta}_{d-r}(\Delta)\right)\\
&=&g_r(\Delta)+g_{r-1}(\partial\Delta)-\sum_{j=2}^r(-1)^{r-j}\left(\binom{d}{r}\tilde{\beta}_{j-1}(\Delta)+ \binom{d}{r-1}\tilde{\beta}_{d-j}(\Delta)\right), \label{g-tilde}
\end{eqnarray}
where the last equality follows from Remark \ref{rem:h''-v2}(2).

\begin{theorem} \label{g-theorem-v2}
Let $\Delta$ be a $(d-1)$-dimensional, connected, orientable $\field$-homology manifold with boundary. 
\begin{enumerate}
\item If $d\geq 4$, then $\big(1, \tilde{g}_1(\complet{\Delta}), \tilde{g}_2(\complet{\Delta})\big)$ is an $M$-vector.
\item If for all vertices $v$ of $\Delta$, $\lk_{\complet{\Delta}}v$ has the WLP over $\field$, then 
$\big(1, \tilde{g}_1(\complet{\Delta}), \tilde{g}_2(\complet{\Delta}), \cdots, \tilde{g}_{\halffloor}(\complet{\Delta})\big)$ is an $M$-vector.
\end{enumerate}
\end{theorem}
\proof
Observe that by definition of $\tilde{g}_r(\complet{\Delta})$,
\begin{eqnarray} \nonumber
\tilde{g}_r(\complet{\Delta})&=&h''_r(\complet{\Delta})-h''_{r-1}(\complet{\Delta})-\left(\binom{d-1}{r-1}\tilde{\beta}_{r-1}(\Delta)+\binom{d-1}{r-2}\tilde{\beta}_{d-r}(\Delta)\right)\\
\nonumber
&=&h''_{d-r}(\complet{\Delta})-h''_{d-r+1}(\complet{\Delta})-\left(\binom{d-1}{d-r+1}\tilde{\beta}_{d-r}(\Delta)+\binom{d-1}{d-r}\tilde{\beta}_{d-(d-r+1)}(\Delta)\right)\\
\label{tilde-g-and-h'}
&=&h''_{d-r}(\complet{\Delta})-h'_{d-r+1}(\complet{\Delta}),
\end{eqnarray}
where the middle step is by Corollary \ref{symmetry} and the last step is by Remark \ref{rem:h''-v2}(1). The rest of the proof follows the proof of \cite[Theorem 5.4(i)]{Murai-Nevo-14}: the only change is that we rely on Theorem~\ref{thm-Gorenstein} that asserts Gorensteinness of $\overline{\field(\complet{\Delta},\Theta)}$ instead of \cite[Theorem 1.4]{Novik-Swartz-09:Gorenstein} that asserts Gorensteinness of the analogous ring associated with a manifold without boundary.
\endproof

\begin{remark} \label{nonnegativity-not-new}
Assume that for all vertices $v$ of $\Delta$, $\lk_{\complet{\Delta}}v$ has the WLP over $\field$ and that  for all {\em boundary} vertices $v$ of $\Delta$, $\lk_{\partial\Delta}v$ has the WLP over $\field$; assume also that $r\leq \lfloor (d-1)/2\rfloor$. Under these assumptions the non-negativity part of Theorem \ref{g-theorem-v2}(2) is not new: the fact that $\tilde{g}_r(\complet{\Delta})\geq 0$ follows from \cite[Theorem 1.5]{Murai-Novik-bdry} (see Theorem \ref{Murai-Novik inequalities}) along with the Poincar\'e--Lefschetz duality and the long exact sequence of $(\Delta,\partial\Delta)$.  For a detailed treatment of the case $d\geq4$ and $r=2$ see the proof of Proposition \ref{links of minimal tilde_g}.
\end{remark}

\section{Applications: K\"uhnel-type bounds}
The results of previous sections can also be used to extend known K\"uhnel-type bounds on the Betti numbers (and their sums) of manifolds without boundary to the case of manifolds with boundary. Deriving such bounds is the goal of this section.

Specifically, Theorem 5.3 in \cite{Murai-15} asserts that if $\Delta$ is a $(d-1)$-dimensional, connected, $\field$-homology manifold without boundary that has $n$ vertices, then $\binom{d+1}{2}\tilde{\beta}_1(\Delta)\leq \binom{n-d}{2}$ as long as $d\geq 4$. Furthermore, Theorem 5.1 in \cite{Murai-15} asserts that if, in addition, all vertex links of $\Delta$ have the WLP over $\field$, then $\binom{d+1}{r+1}\tilde{\beta}_r(\Delta)\leq \binom{n-d-1+r}{r+1}$ for all $r\leq\halffloor-1$. (The conjecture that for $r\leq \lfloor d/2\rfloor-1$ and for an arbitrary $(d-1)$-dimensional simplicial manifold $\Delta$ with $n$ vertices, the inequality $\binom{d+1}{r+1}\tilde{\beta}_r(\Delta)\leq \binom{n-d-1+r}{r+1}$ holds is due to K\"uhnel \cite[Conjecture 18]{Lutz-05}.) In the special case of {\em orientable} $\field$-homology manifolds without boundary the same results were proved in \cite[Theorem 5.2]{Novik-Swartz-09:Socles} and \cite[Theorem 4.3]{Novik-Swartz-09:DS}, respectively. An easy adaptation of proofs from \cite{Novik-Swartz-09:Socles,Novik-Swartz-09:DS} combined with our results from 
the previous sections leads to the following extension. We do not know if this extension also holds in the non-orientable case.

\begin{theorem} \label{thm:Kuhnel-indiv}
Let $\Delta$ be a $(d-1)$-dimensional, connected, orientable $\field$-homology manifold with boundary, and assume that $f_0(\Delta)=n$.
\begin{enumerate}
\item If $d\geq 4$, then $\binom{d}{2}\tilde{\beta}_1(\Delta)+\binom{d}{1}\tilde{\beta}_{d-2}(\Delta)\leq \binom{n-d+1}{2}$. If equality holds, then $\Delta$ is $2$-neighborly and has no interior vertices.
\item If for all vertices $v$ of $\Delta$, $\lk_{\complet{\Delta}}v$ has the WLP over $\field$, then 
\[\binom{d}{r+1}\tilde{\beta}_r(\Delta)+\binom{d}{r}\tilde{\beta}_{d-r-1}(\Delta)\leq \binom{n-d+r}{r+1} \quad \mbox{for all } r\leq\halffloor-1.
\]
If equality holds, then $\Delta$ is $(r+1)$-neighborly and has no interior faces of dimension $\leq r-1$.
\end{enumerate}
\end{theorem}
\proof Since the proof is very similar to that of \cite[Theorem 4.3]{Novik-Swartz-09:DS}, we omit some of the details. Fix an integer $r$:  $r=1$ for part 1 and any $r\leq\halffloor-1$ for part 2. It follows from Theorem \ref{g-theorem-v2} that  $\tilde{g}_{r+1}(\complet{\Delta})$ is nonnegative. Hence
\begin{eqnarray*}
0 &\leq &
h''_{r+1}(\complet{\Delta})-h''_{r}(\complet{\Delta})-\left(\binom{d-1}{r}\tilde{\beta}_{r}(\Delta)+\binom{d-1}{r-1}\tilde{\beta}_{d-r-1}(\Delta)\right)\\
&\stackrel{\mbox{\tiny{by Remark \ref{rem:h''-v2}(1)}}}{=}&h'_{r+1}(\complet{\Delta})-h''_{r}(\complet{\Delta})-\left(\binom{d}{r+1}\tilde{\beta}_{r}(\Delta)+\binom{d}{r}\tilde{\beta}_{d-r-1}(\Delta)\right).
\end{eqnarray*}
We conclude that  $\binom{d}{r+1}\tilde{\beta}_{r}(\Delta)+\binom{d}{r}\tilde{\beta}_{d-r-1}(\Delta) \leq h'_{r+1}(\complet{\Delta})-h''_{r}(\complet{\Delta})$. Thus, to complete the proof, it suffices to show that $h'_{r+1}(\complet{\Delta})-h''_{r}(\complet{\Delta}) \leq \binom{n-d+r}{r+1}$ and that if equality holds then $\complet{\Delta}$ is $(r+1)$-neighborly. (The latter condition implies that $\Delta$ is $(r+1)$-neighborly and that all faces of $\Delta$ of cardinality $\leq r$ are in the link of $v_0$, and hence that they are boundary faces.)

Indeed, since $f_0(\Delta)=n$, $h'_1(\complet{\Delta})=n-d+1$. Macaulay's theorem applied to $\field(\complet{\Delta},\Theta)$, then shows that $h'_{r+1}(\complet{\Delta})=\binom{x+1}{r+1}$ for some real number $x \leq n-d+r$. 
Another application of Macaulay's theorem, this time to $\field(\complet{\Delta},\Theta)/(\Soc \field(\complet{\Delta},\Theta))_r$, yields that $h'_{r+1}(\complet{\Delta})\leq (h''_r(\complet{\Delta}))^{\langle r+1 \rangle}$, and hence that $h''_r(\complet{\Delta})\geq\binom{x}{r}$. Therefore, $h'_{r+1}(\complet{\Delta})-h''_{r}(\complet{\Delta}) \leq \binom{x}{r+1} \leq \binom{n-d+r}{r+1}$, as desired. Furthermore, if $h'_{r+1}(\complet{\Delta})-h''_{r}(\complet{\Delta}) = \binom{n-d+r}{r+1}$, then $\dim_\field \field\big(\complet{\Delta},\Theta\big)_{r+1}=h'_{r+1}(\complet{\Delta})=\binom{n-d+r+1}{r+1}$, which, in turn, implies that $\complet{\Delta}$ is $(r+1)$-neighborly.
\endproof

\begin{corollary} \label{cor:Brehm-Kuhnel-type-bd}
Let $\Delta$ be a $(d-1)$-dimensional, connected, orientable $\field$-homology manifold with boundary, and assume that $f_0(\Delta)=n$. 
\begin{enumerate}
\item If $d\geq 4$, then $\tilde{\beta}_1(\Delta)\leq \binom{n-d+1}{2}/\binom{d}{2}$. In particular, if $\tilde{\beta}_1(\Delta)\neq 0$, then $n\geq 2d-1$.
\item If for all vertices $v$ of $\Delta$, $\lk_{\complet{\Delta}}v$ has the WLP over $\field$, then $\tilde{\beta}_r(\Delta)\leq \binom{n-d+r}{r+1}/\binom{d}{r+1}$ for all $r\leq\halffloor-1$. Consequently, if $\tilde{\beta}_r(\Delta)\neq 0$, then $n\geq 2d-r$. Similarly, if both $\tilde{\beta}_r(\Delta)$ and $\tilde{\beta}_{d-r-1}$ are non-vansishing, then $n\geq 2d-r+1$.
\end{enumerate}
\end{corollary}
The bounds on the number of vertices in the above corollary are similar in spirit to the bounds established by Brehm and K\"uhnel \cite[Theorem B]{Brehm-Kuhnel-87} on the number of vertices that an $(r-1)$-connected, but not $r$-connected closed PL manifold must have. 

\begin{example} K\"uhnel \cite{Kuhnel-86} (see also \cite{Kuhnel-Lassmann-96}) constructed for every $d\geq 3$, a $(d-1)$-dimensional handle, orientable or not depending on the parity of $d$, with exactly $2d-1$ vertices. (For instance, when $d=3$, this gives a unique $5$-vertex triangulation of the M\"obius band.) His construction thus provides a family of connected, orientable over $\Z/2\Z$ manifolds with boundary that have non-vanishing $\tilde{\beta}_1$ and achieve equalities in both statements of Corollary \ref{cor:Brehm-Kuhnel-type-bd}(1). 
\end{example}

We now turn to K\"uhnel-type bounds on certain weighted sums of Betti numbers. It was conjectured by K\"uhnel \cite[Conjecture B]{Kuhnel-95} and proved in \cite[Theorem 4.4]{Novik-Swartz-09:Socles} (see also \cite[Theorem 7.6]{Novik-98}) that if $\Lambda$ is a $2k$-dimensional, orientable $\field$-homology manifold without boundary, then $(-1)^k\big(\tilde{\chi}(\Lambda)-1\big)\leq\binom{f_0(\Lambda)-k-2}{k+1}/\binom{2k+1}{k+1}$. In fact, the proof showed that the same upper bound applies to $\tilde{\beta}_k(\Lambda)+\tilde{\beta}_{k-1}(\Lambda)+2\sum_{i=0}^{k-2}\tilde{\beta}_i(\Lambda)$. The methods of \cite{Novik-98,Novik-Swartz-09:Socles} combined with our results from Sections \ref{sec:Gorenstein} and \ref{sec:h''-and-socles} lead to the following extension of this result to manifolds with boundary.

\begin{theorem} \label{with-boundary-weighted-Betti}
Let $\Delta$ be a connected, orientable, $\field$-homology manifold with boundary. If $\Delta$ is $2k$-dimensional and has $n$ vertices, then 
\begin{equation}  \label{eq:weighted-Betti}
\tilde{\beta}_k(\Delta)+\sum_{i=2}^k\frac{\binom{n-k-1}{k+1}}{{\binom{2k+1}{k+1}}{\binom{n-2k-1+i }{i}}} \cdot \left(\binom{2k}{i}\tilde{\beta}_{i-1}(\Delta) +\binom{2k}{i-1}\tilde{\beta}_{2k+1-i}(\Delta)\right)\leq \frac{\binom{n-k-1}{k+1}}{\binom{2k+1}{k+1}}.
\end{equation}
Equality holds if and only if $\Delta$ is $(k+1)$-neighborly and has no interior faces of dimension $\leq k-1$.
\end{theorem}
Examples that achieve equality include $(k+1)$-neighborly triangulations of closed manifolds of dimension $2k$ with one vertex removed. Before proving Theorem \ref{with-boundary-weighted-Betti} we discuss some of its consequences.

\begin{corollary} \label{with-boundary-middle-Betti}
Let $\Delta$ be a connected, orientable, $\field$-homology manifold with boundary. Assume $\Delta$ is $2k$-dimensional and has $n$ vertices. Then 
\begin{enumerate}
\item $\tilde{\beta}_k(\Delta)\leq \frac{\binom{n-k-1}{k+1}}{\binom{2k+1}{k+1}}$. In particular, if $\tilde{\beta}_k(\Delta)\neq 0$, then $n\geq 3k+2$.
\item If $n\geq 3k+2$, then $\tilde{\beta}_k(\Delta)+\sum_{i=2}^{k-1}\tilde{\beta}_{i-1}(\Delta) \leq \frac{\binom{n-k-1}{k+1}}{\binom{2k+1}{k+1}}$.
\item If $n\geq 4k+2$, then $\sum_{i=2}^{k+1} \tilde{\beta}_{i-1}(\Delta)\leq \frac{\binom{n-k-1}{k+1}}{\binom{2k+1}{k+1}}$.
\end{enumerate}
\end{corollary}

To derive parts 2 and 3 of Corollary \ref{with-boundary-middle-Betti} from Theorem \ref{with-boundary-weighted-Betti}, use  routine computations with binomial coefficients to show that if $n\geq 4k+2$, then the coefficient of $\tilde{\beta}_{i-1}$ in \eqref{eq:weighted-Betti} is at least $1$ for all $i\leq k+1$, while if $n\geq 3k+2$, then such a coefficient is $\geq 1$ for all $i\leq k+1$ except possibly for $i=k$. The proof of Theorem \ref{with-boundary-weighted-Betti} is very similar to that of \cite[Theorem 4.4]{Novik-Swartz-09:Socles}, and so we only sketch the main details.

\smallskip\noindent {\it Proof of Theorem \ref{with-boundary-weighted-Betti}} (Sketch): \ 
Let $N_p:=\binom{f_0(\complet{\Delta})-(2k+1)+p-1}{p}=\binom{n-2k-1+p}{p}$. In particular, $N_{k+1}-N_k=\binom{n-k-1}{k+1}$.

Applying Macaulay's theorem to $\field(\complet{\Delta},\Theta)/(\Soc \field(\complet{\Delta},\Theta))_i$, yields that 
\[
h'_{i+1}(\complet{\Delta})\leq (h''_i(\complet{\Delta}))^{\langle i+1 \rangle} \leq \frac{N_{i+1}}{N_i}h''_i(\complet{\Delta}) =
\frac{N_{i+1}}{N_i} \left(h'_i(\complet{\Delta})-\dim_\field \big(\Soc  \field(\complet{\Delta},\Theta)\big)_i\right)
\quad \mbox{for all }i\leq 2k.
\]
Iterating this process (see the proof of  \cite[Theorem 4.4]{Novik-Swartz-09:Socles} for more details), we obtain that
\begin{eqnarray}  \label{h'-h'}
h'_{k+1}(\complet{\Delta})-h'_{k}(\complet{\Delta})&\leq&\\
\nonumber
(N_{k+1}-N_k)&-&
\left[\frac{N_{k+1}}{N_k} \dim_\field \big(\Soc  \field(\complet{\Delta},\Theta)\big)_k+
\frac{N_{k+1}-N_k}{N_k}\sum_{i=2}^{k-1} \frac{N_k}{N_i}\dim_\field \big(\Soc \field(\complet{\Delta},\Theta)\big)_i
\right].
\end{eqnarray}
On the other hand, since 
$h'_i(\complet{\Delta})=h''_i(\complet{\Delta})+\dim_\field \big(\Soc  \field(\complet{\Delta},\Theta)\big)_i$ and since $h''_{k+1}(\complet{\Delta})=h''_{k}(\complet{\Delta})$ by Corollary \ref{symmetry}, it follows that
\begin{equation} \label{Soc-Soc}
h'_{k+1}(\complet{\Delta})-h'_{k}(\complet{\Delta})=\dim_\field \big(\Soc \field(\complet{\Delta},\Theta)\big)_{k+1}-\dim_\field \big(\Soc \field(\complet{\Delta},\Theta)\big)_k.
\end{equation}
Combining equations \eqref{h'-h'} and \eqref{Soc-Soc}, we conclude that
\[
\dim_\field \big(\Soc \field(\complet{\Delta},\Theta)\big)_{k+1}+\frac{N_{k+1}-N_k}{N_k}\sum_{i=2}^{k} \frac{N_k}{N_i}\dim_\field \big(\Soc \field(\complet{\Delta},\Theta)\big)_i \leq N_{k+1}-N_k.
\]
Substituting expressions for the dimensions of graded components of the socle from Remark \ref{rem:h''-v2}(1) and, in particular, noting that 
$\dim_\field \big(\Soc \field(\complet{\Delta},\Theta)\big)_{k+1}=\binom{2k+1}{k}\tilde{\beta}_k(\Delta)$, yields the inequality. The treatment of equality case is almost identical to that in \cite[Theorem 4.4]{Novik-Swartz-09:Socles} and is omitted.
\endproof

\section{Equality}
  In this section we examine the combinatorial and topological consequences of some of the known inequalities for $f$-vectors of homology manifolds with boundary when they are sharp.  This includes a discussion of a connection between three lower bound theorems for manifolds,  PL-handle decompositions, and surgery. Along the way we propose several problems.  
  
	The right-hand side of Theorem \ref{known-h'-h''}(3) makes sense for any simplicial complex $\Delta.$ So we define
$$\bar{h}''_i(\Delta):= h_i(\Delta)-\binom{d}{i}\sum_{j=1}^{i}(-1)^{i-j}\tilde{\beta}_{j-1}(\Delta) \qquad \forall ~0\leq i\leq d-1.$$
It turns out that for homology manifolds with boundary, or more generally Buchsbaum complexes, $\bar{h}''_i \ge 0$ \cite[Section 3]{Novik-Swartz-09:Socles}. In fact, $\bar{h}''$-numbers of Buchsbaum complexes have an algebraic interpretation, see \cite[Theorem 1.2]{Murai-Novik-Yoshida}.  Murai and Nevo determined the combinatorial implications of $\bar{h}''_i=0.$  To state this we recall that a homology manifold with  boundary is {\em $i$-stacked} if it contains no interior faces of codimension $i+1$ or more.  A homology manifold without boundary is {\em $i$-stacked} if it is the boundary of an $i$-stacked homology manifold with boundary. As is customary, for both homology manifolds with or without boundary we will generally shorten $1$-stacked to just {\em stacked.}

\begin{theorem} \label{h''=0} {\rm \cite[Theorem 3.1]{Murai-Nevo-14}}
Let $\Delta$ be a $(d-1)$-dimensional homology manifold with boundary, $1 \le i \le d-1$, and $d \ge4.$   Then $\bar{h}_i''(\Delta)=0$ if and only if $\Delta$ is $(i-1)$-stacked.
\end{theorem}

 Murai and Nevo further noted that with the same hypotheses, $\bar{h}''_i=0$ also implied that $\tilde{\beta}_j = 0$ for all $j \ge i$  \cite[Corollary 3.2]{Murai-Nevo-14}.   When $\Delta$ is a PL-manifold with boundary the above combinatorial restriction has an even stronger topological implication in terms of the PL-handle decomposition of $\|\Delta\|.$  In order to describe this we review handle decompositions of PL-manifolds.  We refer the reader to Rourke and Sanderson \cite{Rourke-Sanderson} for definitions and results concerning PL-manifolds.

Let $B$ be a $(d-1)$-dimensional PL-ball decomposed as a product $B = B^s \times B^t,$ where $B^s$ and $B^t$ are PL-balls of dimensions $s$ and $t$ respectively.  Hence, 
$$\partial B = (\partial B^s \times B^t) \displaystyle\bigcup_{\partial B^s  \times \partial B^t} (B^s \times \partial B^t).$$
\noindent   Now let $X$ be a $(d-1)$-dimensional PL-manifold with boundary.  We say that $X'$ is obtained from $X$ by \emph{adding a PL-handle of index $s$} if 
$ X'$ is the union of $X$ and $B$ and, in addition, the intersection of $X$ and $B$ is contained in the boundary of $X$ and equals  $\partial B^s \times B^t.$
   For instance, adding a disjoint ball to $X$ is adding a PL-handle of index $0.$   A \emph{PL-handle decomposition of $X$ } is a sequence of $(d-1)$-dimensional PL-manifolds
$$X_1 \subseteq X_2 \subseteq \cdots \subseteq  X_r = X$$
such that $X_1$ is a PL-ball and for $1 \le j \le r-1$ each $X_{j+1}$ is obtained from $X_j$ by adding a PL-handle.

The following result first appeared as a remark in Section 6 of \cite{Swartz-14}. We include it here for completeness.
\begin{theorem} \label{h''=0 -> handle decomposition}
Suppose $\Delta$ is a $(d-1)$-dimensional PL-manifold with boundary, $d \ge4,$  and $\bar{h}_i''(\Delta)=0$ for some $1 \le i \le d-1$. Then $\|\Delta\|$ has a PL-handle decomposition using handles of index less than $i$.
\end{theorem}

\begin{proof}
Let $\Delta''$ be the second barycentric subdivision of $\Delta.$  For each nonempty face $F$ of $\Delta,$ let $v_F$ be the vertex in $\Delta''$ which represents $F.$  The star of $v_F$ in $\Delta''$ is a PL-ball and every facet of $\Delta''$ is contained in exactly one such star.    Now order the interior faces $F$ of $\Delta,~F_1, F_2, \dots, F_r$ so that all of the codimension zero faces (the facets) of $\Delta$ come first, then the interior faces of codimension one, etc.   Finally, set $X_j = \bigcup ^j_{k=1} \st_{\Delta''} v_{F_k}.$  Thus, for $j \le f_{d-1}(\Delta), ~X_j$ is a disjoint union of $j$ PL-balls.  By \cite[Proposition 6.9]{Rourke-Sanderson} and the discussion that precedes it, $X_1 \subseteq \cdots \subseteq X_r$   is a handle decomposition of $\|\Delta\|$ with a collar of the boundary removed.  Furthermore, the index of the handle attached to go from $X_j$ to $X_{j+1}$ is the codimension of $F_{j+1}.$    Since removing a collar does not change the PL-homeomorphism type of a complex, $X_1 \subseteq \cdots \subseteq X_r$ is the handle decomposition of a PL-manifold which is PL-homeomorphic to $\|\Delta\|.$  Theorem \ref{h''=0} completes the proof.  
\end{proof}

What about the converse?

\begin{problem}
Suppose $X$ is a $(d-1)$-dimensional PL-manifold with boundary that has  a PL-handle decomposition using handles of index less than $i$ for some $1 \le i \le d-1.$  Is there a PL-triangulation $\Delta$ of $X$ such that $\bar{h}''_i(\Delta) = 0?$
\end{problem}

  For $i = 1,2,$ and $d-1$ the answer to the above question is yes. If $X$ has a PL-handle decomposition involving only handle additions of index zero, then $X$ is a disjoint union of PL-balls.    Hence a disjoint union of $(d-1)$ simplices triangulates $X$ and has $\bar{h}''_1=0.$  For $i=2$ we first observe that if $X$ has a PL-handle decomposition using handles of index zero or one,  then  $X$ is a handlebody and all of these have stacked triangulations, which are precisely triangulations with $\bar{h}''_2=0$. (This observation is any easy consequence of, say, \cite[Theorem 3.11]{DattaMurai}.)   For the last case we first note that any $(d-1)$-dimensional PL-manifold with nonempty boundary has a PL-handle decomposition which does not have $(d-1)$-handles \cite[Corollary 6.14 (ii)]{Rourke-Sanderson}.  On the other hand, every such space has a PL-triangulation with no interior vertices \cite[Theorem 1]{Bing}.
  

The above theorems and problems have close analogs for  manifolds without boundary.  Suppose $X'$ is obtained from $X$ by adding an $s$-handle.  Then the boundary of $X'$ is a $(d-2)$-dimensional PL-manifold without boundary and is obtained from $\partial X$  by removing a copy of $\partial B^s \times B^t$ from $\partial X$ and replacing it with $B^s \times \partial B^t$ along the common boundary $\partial B^s \times \partial B^t.$  Such an operation is called an {\em $(s-1)$-surgery} on $\partial X$ and we call $s-1$ the {\em index} of the surgery.     We denote such a surgery operation by $\partial X \Rightarrow \partial X'.$ So, if $X$ has a handle decomposition $B^{d-1} = X_1 \subseteq X_2 \subseteq \cdots \subseteq X_r = X,$ then $\partial X$ has a surgery sequence $S^{d-2}=\partial X_1 \Rightarrow \cdots \Rightarrow \partial X_r=\partial X.$  From the $g$-vector point of view the connection between these two is given by the following theorem of Murai--Nevo.  Note that if $\Delta$ is a $(d-1)$-dimensional, connected, orientable homology manifold without boundary, then eq.~\eqref{g-tilde} reduces to  $\tilde{g}_r(\Delta)=g_r(\Delta)-\binom{d+1}{r}\sum_{j=1}^r(-1)^{r-j}\tilde{\beta}_{j-1}(\Delta)$.  We use the same equation to define $\tilde{g}_r$ for {\em all} $(d-1)$-dimensional homology manifolds without boundary.

\begin{theorem}  {\rm\cite{Murai-Nevo-14}}  
Let $\Delta$ be a $(d-1)$-dimensional homology manifold and $d \ge 4.$   
  \begin{enumerate}
    \item  If $\partial \Delta \neq \emptyset$ and $\bar{h}''_i(\Delta)=0$ for some $i\leq(d-1)/2$, then $\tilde{g}_i(\partial \Delta) = 0.$
    \item  If $\partial \Delta = \emptyset$,  the links of the vertices of $\Delta$ have the WLP, and $\tilde{g}_i(\Delta)=0$ for some $1\leq i \leq (d-1)/2$, then $\Delta$ is $(i-1)$-stacked.
  \end{enumerate}
\end{theorem}

\noindent  In combination with Theorem \ref{h''=0 -> handle decomposition} two natural questions are:

\begin{problem}
Let  $\Delta$ be a $(d-1)$-dimensional PL-manifold without boundary, $d \ge 4,$ and $2 \le i \le (d-1)/2.$
\begin{enumerate}
 \item  If   $\tilde{g}_i(\Delta) =0$, does $\|\Delta\|$ have a surgery sequence beginning with $S^{d-1}$ and using surgeries whose indices are less than $i-1?$ 
 \item  Suppose $X$ is a $(d-1)$--dimensional PL-manifold with a surgery sequence $X = X_1 \Rightarrow \cdots \Rightarrow X_r = \|\Delta\|$ whose indices are less than $i-1.$    Does $X$ have a PL-triangulation $\Delta$ with $\tilde{g}_i(\Delta) = 0?$
 \end{enumerate}
\end{problem}
\noindent Note that for $i=2$ the answer to the first part of the problem is yes; see the discussion preceding Theorem \ref{global minimal g2}.

In \cite{Murai-Novik-bdry} Murai and Novik considered a different invariant of the $f$-vector.  Let  $\Delta$ be a homology manifold  and define $f_i(\Delta, \partial \Delta)$ to be the number of interior $i$-dimensional faces.  If $\Delta$ has a nonempty boundary, $f_{-1}(\Delta, \partial \Delta)=0$ as the empty set is no longer an interior face.   Now define all of the other invariants, such as $h_i(\Delta,\partial \Delta)$ and  $g_i(\Delta, \partial \Delta)$ by using $f_i(\Delta, \partial \Delta)$  instead of $f_i(\Delta).$  For example,

$$g_1(\Delta, \partial \Delta) = h_1(\Delta, \partial \Delta) - h_0(\Delta, \partial \Delta) = f_0(\Delta, \partial \Delta) - (d+1) f_{-1}(\Delta, \partial \Delta),$$
and
$$g_2(\Delta, \partial \Delta) = h_2(\Delta, \partial \Delta) - h_1(\Delta, \partial \Delta) = f_1(\Delta, \partial \Delta) - d ~ f_0(\Delta, \partial \Delta) + \binom{d+1}{2} f_{-1}(\Delta, \partial \Delta).$$
\noindent Among Murai--Novik's results is the following.

\begin{theorem} \label{Murai-Novik inequalities} {\rm\cite{Murai-Novik-bdry}}
Let $\Delta$ be a $(d-1)$-dimensional $\field$-homology manifold and $d \ge 4.$ 
 \begin{enumerate}
   \item \label{Murai-Novik 1 or 2} For $i = 1$ or $2,~ g_i(\Delta, \partial \Delta)  \ge \binom{d+1}{i} \displaystyle\sum^i_{j=1} (-1)^{i-j} \tilde{\beta}_{j-1}(\Delta, \partial \Delta).$
   \item  If the links of the vertices of $\Delta$ satisfy the WLP and $1 \le i \le d/2,$ then $$ g_i(\Delta, \partial \Delta)  \ge \binom{d+1}{i} \displaystyle\sum^r_{j=1} (-1)^{i-j} \tilde{\beta}_{j-1}(\Delta, \partial \Delta).$$
 \end{enumerate}
\end{theorem}

\noindent  In fact, Theorem \ref{Murai-Novik inequalities}(\ref{Murai-Novik 1 or 2}) holds for the larger class of normal pseudomanifolds with boundary and  Betti numbers replaced with the more subtle $\mu$-invariant of Bagchi and Datta.  See \cite[Theorem 7.3]{Murai-Novik-bdry} for details.

Now we consider the implications of equality in Theorem \ref{Murai-Novik inequalities}.  Suppose $\Delta$ satisfies the hypotheses of Theorem \ref{Murai-Novik inequalities}. Then $g_1(\Delta, \partial \Delta) = (d+1) \tilde{\beta}_0(\Delta, \partial \Delta)$  if and only if  all of the vertices of every component of $\Delta$ which has boundary are  on the boundary, and every component of $\Delta$ which does not have boundary is the boundary of a $d$-simplex.  In particular, if $\Delta$ is also a PL-manifold, then its components with boundary have {\em no}  further topological restrictions \cite{Bing},  while the components without boundary must be PL-spheres. The situation for general homology manifolds is less clear.  For instance, suppose $X$ is the suspension of $\R P^3.$ Now remove an open ball whose closure does not include the suspension points of $X$  and call the resulting space $Y.$  Then $Y$ is a $\Q$-homology ball and excision applied to homology with integer coefficients around the suspension points of $X$ shows that in any triangulation $\Delta$ of $Y$ the suspension points of $X$ must be vertices of $\Delta$ and are not on the boundary of $\Delta.$  

\begin{problem}
  What are the topological restrictions imposed on $\field$-homology manifolds by the relation $g_1(\Delta, \partial \Delta) = (d+1) \tilde{\beta}_0(\Delta, \partial \Delta)?$
\end{problem}

For $\field$-homology manifolds which satisfy equality in Theorem \ref{Murai-Novik inequalities}(\ref{Murai-Novik 1 or 2}) with $i=2$, Murai and Novik gave a local combinatorial description in terms of the links of the vertices.  If $\Delta$ does satisfy \ref{Murai-Novik inequalities}(\ref{Murai-Novik 1 or 2})  with equality and $i=2$, we say that $\Delta$ {\em has minimal $g_2.$} Before stating their result we review the  operations and properties of connected sum and handle addition.  

    Let $\Delta_1$ and $\Delta_2$ be $(d-1)$-dimensional complexes with disjoint vertex sets. Suppose $F_1$ and $F_2$ are facets of $\Delta_1$ and $\Delta_2$ respectively and $\phi:F_1 \to F_2$ is a bijection.   The {\em connected sum of $\Delta_1$ and $\Delta_2$ along $\phi$} is the complex obtained by identifying all faces $\sigma \subseteq F_1$ with $\phi(\sigma) \subseteq F_2$ and then removing the identified facet $F_1 \equiv F_2.$   The resulting complex is denoted by $\Delta_1 \# \Delta_2,$ or by $\Delta_1 \#_\phi \Delta_2$ if we need to specify $\phi.$   To define handle addition we  suppose $F_1$ and $F_2$ are both facets of a {\em single} component of a complex $\Delta$ and $\phi$ is still a bijection between them.  Now make the same identifications and facet removal as in the connected sum.  As long as the graph distance between $v$ and $\phi(v)$ is at least three for all $v \in F_1,$ the result is a simplicial complex which we denote by $\Delta^{\#},$ or by $\Delta^{\#}_\phi$ if we need to specify $\phi.$  If $F_1$ and $F_2$ are in the same complex, but distinct components we rename the components as distinct complexes and use the connected sum notation.  Note that if $\Delta_1$ and $\Delta_2$ are PL-manifolds without boundary then $\|\Delta_1\#\Delta_2\|$ and $\|\Delta_1^{\#}\|$ are produced from $\|\Delta_1\cup\Delta_2\|$ and $\|\Delta_1\|$ respectively by $0$-surgery.
    
    As pointed out in \cite[Lemma 7.7]{Murai-Novik-bdry} the connected sum of a $\field$-homology ball and a $\field$-homology sphere of the same dimension is a $\field$-homology ball whose boundary is the same as the boundary of the original homology ball.  Similarly, the connected sum of two $\field$-homology spheres of the same dimension is another $\field$-homology sphere. On the other hand, the connected sum of two  $\field$-homology balls of the same dimension is neither a  $\field$-homology ball nor a sphere. Thus, if $\Delta_1$ and $\Delta_2$ are $\field$-homology manifolds of the same dimension, then $\Delta_1 \#_\phi~\Delta_2$ is a $\field$-homology manifold if and only if for each vertex $v$ in the identified facet at least one of $v$ or $\phi(v)$ is an interior vertex.  A similar statement holds for $\Delta^\#.$    Lastly, we observe that the boundary of $\Delta_1 \# \Delta_2$ is the disjoint union of the boundaries of $\Delta_1$ and $\Delta_2.$   Similarly, the boundary of $\Delta^\#$ equals the boundary of $\Delta.$ 
    
    Both connected sum and handle addition introduce a missing facet into the resulting complex.  A {\em missing facet} in a $(d-1)$-dimensional complex $\Delta$ is a subset $F$ of cardinality $d$ of the vertices such that $F \notin \Delta,$ but every proper subset of $F$ is a face of $\Delta.$ For future inductive purposes we observe that connected sum and handle addition strictly increase the number of missing facets.   In homology manifolds missing facets characterize the connected sum and handle addition operations.
    
    \begin{proposition} \label{missing facets}
      Suppose $\Delta$ is a $(d-1)$-dimensional homology manifold, $d \ge 4$, and $F$ is a missing facet of $\Delta.$  Then either $\Delta$ is a connected sum of homology manifolds, or $\Delta$ is the result of a handle addition on a homology manifold.  
          \end{proposition}
    
    \begin{proof}
		 Consider $\complet{\Delta}$. In $\complet{\Delta}$, the links of all of the vertices of $F$ are homology spheres, and so Alexander duality implies that the boundary of $F$ is locally two-sided. The argument of \cite[Lemma 3.2]{Basak-Swartz} then shows that $\|\partial F\|$ is two-sided in $\|\complet{\Delta}\|$. Now, if $\partial\Delta=\emptyset$, in which case $\Delta=\complet{\Delta}$, and $\|\partial F\|$ is two-sided in $\|\Delta\|$, the above statement is known; for a very detailed treatment see \cite[Lemma 3.3]{Bagch-Datta-LBT}. So assume $\partial\Delta\neq\emptyset$. Cut $\complet{\Delta}$ along the boundary of $F$ and fill in the two missing $(d-1)$-faces that result from $F$. We obtain either a connected complex or two disjoint complexes one of which contains $v_0$ --- the singular vertex of $\complet{\Delta}$. Thus we can write $\complet{\Delta}=\complet{\Gamma}^\#$ or $\complet{\Delta} = \complet{\Delta}_1 \# \Delta_2$, where $v_0$ is in $\complet{\Delta}_1$. Removing $v_0$ allows us to write $\Delta= \Gamma^\#$ or $\Delta= \Delta_1 \# \Delta_2$.  
		
      We consider the case $\Delta = \Delta_1 \#_\phi \Delta_2,~\phi:F_1 \to F_2,$ as the handle addition case is virtually identical.  All that remains is to show that $\Delta_1$ and $\Delta_2$ are homology manifolds.  The vertices of $\Delta_1$ and $\Delta_2$ which are not in $F_1$ or $F_2$   have links which are simplicially isomorphic to their image in $\Delta, $ and hence are homology balls or spheres.  Now suppose that $v \in F_1$ and let $x$ be its image in $F.$  If the link of $x$ in $\Delta$ was a homology sphere, then the links of $v$ in $\Delta_1$ and $\phi(v)$ in $\Delta_2$ are also homology spheres.  If the link of $x$ in $\Delta$ was a homology ball, then in $\complet{\Delta}$ the link of $x$ is a homology sphere $\Gamma$ which is the link of $x$ in $\Delta$ with its boundary coned off.   Since $F-x$ is a missing facet in $\Gamma,$ we can write $\Gamma = \Gamma_1 \# \Gamma_2,$ where each $\Gamma_i$ is a homology sphere and the identified facet is  $F-x.$  
 The link of $v$ in $\Delta_1$ is then $\Gamma_1$ with the vertex $v_0$ removed and hence is a homology ball, 
while the link of $\phi(v)$ in $\Delta_2$ is $\Gamma_2$ and is therefore a homology sphere.  Finally, to see that the boundaries of $\Delta_1$  and $ \Delta_2$ are (possibly empty) $(d-2)$-dimensional homology manifolds we simply recall that the boundary of $\Delta = \Delta_1 \# \Delta_2$ is equal to the disjoint union of the boundaries of $\Delta_1$ and $\Delta_2.$            
   \end{proof}

We now list several procedures which result in a $\Delta$ that has minimal $g_2.$  
All of the proofs are routine applications of the definitions and/or an expected Mayer-Vietoris sequence.   For instance, the proof of the third part relies on the following observations: $\tilde{\beta}_1\big(\Delta_1\#\Delta_2, \partial(\Delta_1\#\Delta_2)\big)=\tilde{\beta}_1(\Delta_1,\partial\Delta_1)+\tilde{\beta}_1(\Delta_2, \partial\Delta_2)$, $f_1\big(\Delta_1\#\Delta_2, \partial(\Delta_1\#\Delta_2)\big)=f_1(\Delta_1,\partial\Delta_1)+f_1(\Delta_2, \partial\Delta_2)-\binom{d}{2}$, and $f_0\big(\Delta_1\#\Delta_2, \partial(\Delta_1\#\Delta_2)\big)=f_0(\Delta_1,\partial\Delta_1)+f_0(\Delta_2, \partial\Delta_2)-d$.
\begin{proposition} \label{minimal g2 properties}
 Let $\Delta$ be a $(d-1)$-dimensional $\field$-homology manifold,  where $d\geq 4$.  
  \begin{enumerate}
    \item 
    If $\Delta$ has no interior edges, then $\Delta$ has minimal $g_2.$  
     \item \label{components}
       $\Delta$ has minimal $g_2$ if and only if each component of $\Delta$ has minimal $g_2.$ 
      \item 
       If $\Delta = \Delta_1 \# \Delta_2$ with $\Delta_1$ and $\Delta_2$ both $\field$-homology manifolds, then $\Delta$ has minimal $g_2$ if and only if $\Delta_1$ and $\Delta_2$ have minimal $g_2$ and at least one of $\Delta_1, \Delta_2$ has no boundary.  
      \item
      If $\Delta = \Gamma^\#$ with $\Gamma$ a $\field$-homology manifold, then $\Delta$ has minimal $g_2$ if and only if $\Gamma$ has minimal $g_2.$
  \end{enumerate}
\end{proposition}

Here is the Murai--Novik restriction on links of vertices in complexes with minimal $g_2.$  In combination with the previous propositions it allows us to describe a global combinatorial characterization of such complexes.  
\begin{theorem}  \label{local minimal g2} {\rm\cite[Section 7]{Murai-Novik-bdry}}
    Let $\Delta$ be a $(d-1)$-dimensional $\field$-homology manifold with $d \ge 4$ and minimal $g_2.$    Then the link of every interior vertex is a stacked sphere.  Furthermore, for every boundary vertex $v$ there exists $m \ge 0$ (which depends on $v$) such that the link of $v$ is of the form 
    $$\T \# \Sm_1 \# \cdots \# \Sm_m,$$
    where $\T$ is a homology ball with no interior vertices and each $\Sm_i$ is the boundary of a $(d-1)$-simplex.    
\end{theorem}

 Recall that  homology manifolds without boundary and minimal $g_2$ are well understood: according to \cite[Theorem 5.3]{Murai-15} (that built on \cite[Theorem 5.2]{Novik-Swartz-09:Socles} and \cite[Theorem 1.14]{Bagchi:mu-vector}, as well as on the notions of $\sigma$- and $\mu$-numbers introduced in \cite{Bagchi-Datta-14}), they are stacked homology manifolds without boundary, which in turn are precisely the elements of the Walkup's class introduced in \cite{Walkup-70} (see also \cite[Section 8]{Kalai-87}). Each such manifold is obtained by starting with several disjoint boundary complexes of the $d$-simplex and repeatedly forming connected sums and/or handle additions. In particular, if $\Delta$ is a stacked homology manifold without boundary, then  $\Delta$ is PL; furthermore,  $\|\Delta\|$ is a sphere, a sphere bundle over $S^1$, or a connected sum of several of these. In view of this and Proposition \ref{minimal g2 properties}(\ref{components}), 
we now concentrate on connected homology manifolds with boundary.  Our goal is to prove the following theorem.

\begin{theorem} \label{global minimal g2}
Let $\Delta$ be a $(d-1)$-dimensional, connected, $\field$-homology manifold  with boundary. Assume further that $\Delta$ has minimal $g_2$ and $d \ge 4.$  Then there is a sequence $\Delta_1 \rightarrow \cdots \rightarrow \Delta_r =\Delta$ such that every $\Delta_i$ has boundary, minimal $g_2,$ and $\Delta_1$ has no interior edges.  Furthermore,   for every $1 \le i \le r-1,~\Delta_{i+1}$ is equal to $\Delta^\#_i$ or $\Delta_i \# \Gamma,$ where $\partial \Gamma = \emptyset$ and $\Gamma$ has minimal $g_2.$  
\end{theorem}

\begin{proof}
 If the link of any vertex is the boundary of a $(d-1)$-simplex, then either $\Delta$ is the boundary of the $d$-simplex or we can remove the vertex and replace its star with a facet.  Repeating this procedure as many times as necessary we can assume that there is no vertex whose link is the boundary of a $(d-1)$-simplex.  The proof now continues  by induction on the number of missing facets.  

First we show that if $\Delta$ has no missing facets, then $\Delta$ has no interior edges and hence $\Delta=\Delta_1$ is the required sequence.  Thus let $e$ be an interior edge with endpoints $v$ and $w.$  There are two cases to consider: (i) either $v$ or $w$ is an interior vertex, say $v$, or (ii) both $v$ and $w$ are boundary vertices. Theorem \ref{local minimal g2} then shows that in the former case, the link of $v$ must be a stacked sphere which by our assumption is not the boundary of the simplex; hence, the link of $v$ is of the form $\Sm_0\#\Sm_1\#\cdots\#_{\phi}\Sm_m$, where $m\geq 1$ and $\Sm_m$ is the boundary of the $(d-1)$-simplex. Similarly, in the latter case, since $v$ is the boundary vertex whose link has the interior vertex $w$, the link of $v $ is $\T \# \Sm_1 \# \cdots \#_{\phi} \Sm_m,$ where $m\geq 1$ and $\Sm_m$ is the boundary of the $(d-1)$-simplex. Thus, in either case the link of $v$ contains a vertex $x$ (e.g., the vertex of $\Sm_m$ that is not in the image of $\phi$) such that the link of the edge $f=\{v,x\}$ is  the boundary of the $(d-2)$-simplex $G$ (the facet of $\Sm_m$ opposite to $x$). Hence $\st f$ is  $f \ast \partial G.$  If $G \notin \Delta$ then we retriangulate $\st f$ by removing $f$ and inserting two new facets  $v \cup G$ and $x \cup G.$  (This is usually called a $(d-2)$ bistellar move.)  The resulting complex is homeomorphic to $\Delta$ but has smaller $g_2.$  This is impossible, so  $G \in \Delta.$  However, $G\in\Delta$ implies that $v \cup G$ or $x \cup G$ is a missing facet of $\Delta$ as  otherwise $\Delta$ contains the boundary of the $d$-simplex $\{v,x\} \cup G.$

Once we know that $\Delta$ has at least one missing facet we can write $\Delta$ as  $\Delta_1 \# \Delta_2$ or $\Gamma^\#$  (see Proposition \ref{missing facets}) and apply Proposition \ref{minimal g2 properties} and the induction hypothesis  along with the known characterization of stacked homology manifolds without boundary to produce the required sequence of complexes.  
 \end{proof}
 
There are no immediately obvious Betti number restrictions on $\Delta$ when $\Delta$ has minimal $g_2.$  However, there are {\em some} topological restrictions.   For instance, let $X$ be an integral homology sphere with nontrivial fundamental group and let $Y$ be $X$ with a small ball removed. If $\Delta$ is a triangulation of $Y,$ then \cite[Theorem 7.3]{Murai-Novik-bdry} (see also \cite{Murai-Novik-fundam-gp}) can be used to show that even though $\tilde{\beta}_1(\Delta, \partial \Delta) = \tilde{\beta}_0(\Delta, \partial \Delta)=0,~g_2(\Delta, \partial \Delta) > 0.$

 \begin{problem}
   What topological restrictions does the above combinatorial decomposition imply for PL-manifolds with boundary that have minimal $g_2?$  What about general homology manifolds with boundary that have minimal $g_2?$
 \end{problem}

 \begin{problem}
 Is there a similar decomposition for $\Delta$ when $\Delta$ has minimal $g_i$ for $i \ge 3?$
 \end{problem}
 
 The last  inequality we consider is  $\tilde{g}_2(\complet{\Delta}) \ge 0.$  As noted in Remark \ref{nonnegativity-not-new},  at least for $d\geq 5$, this inequality is implied by Theorem \ref{Murai-Novik inequalities}.  In fact, for connected orientable homology manifolds with boundary, $\tilde{g}_2(\complet{\Delta}) \ge 0$ can be a strictly weaker statement than the Murai--Novik inequality in Theorem \ref{Murai-Novik inequalities}.  So it is reasonable to expect a stronger conclusion from $\tilde{g}_2(\complet{\Delta})=0.$ 
 When a connected orientable $\field$-homology manifold $\Delta$ satisfies $\tilde{g}_2(\complet{\Delta}) = 0$ we will say $\complet{\Delta}$ has {\em minimal $\tilde{g}_2.$}  (Note that for homology manifolds without boundary, having minimal $\tilde{g}_2$ and having minimal $g_2$ are equivalent properties.) We begin by noting how connected sum and handle addition interact with minimal $\tilde{g}_2.$   The proofs are the usual applications of Mayer-Vietoris and  the definitions.
 \begin{proposition} \label{connected sum and handle addition for g-tilde}
 Let $\Delta_1, \Delta_2,$ and $\Gamma$ be $(d-1)$-dimensional, connected, orientable $\field$-homology manifolds with boundary. 
   \begin{enumerate}
       \item 
       $\complet{\Gamma}^\#$ has minimal $\tilde{g}_2$ if and only if $\complet{\Gamma}$ has minimal $\tilde{g}_2.$
       \item
       Suppose that the connected sum of $\Delta_1$ and $\Delta_2$  is a $\field$-homology manifold. Then the completion of $\Delta_1 \# \Delta_2$ has minimal $\tilde{g}_2$ if and only if $\complet{\Delta}_1$ and $\complet{\Delta}_2$ have minimal $\tilde{g}_2$ and at least one of $\Delta_1$ or $\Delta_2$ has no boundary.            
   \end{enumerate}
 \end{proposition}
 
 Like in the previous two cases, the key to analyzing complexes with minimal $\tilde{g}_2$ involves understanding the links of vertices.  
 
 \begin{proposition} \label{links of minimal tilde_g}
 Let $\Delta$ be a $(d-1)$-dimensional, connected, orientable $\field$-homology manifold with boundary such that $d \ge 4$ and the completion of $\Delta$ has minimal $\tilde{g}_2.$   Then the link of every interior vertex of $\Delta$ is a stacked sphere while the link of every boundary vertex is a stacked sphere with one vertex removed.
 \end{proposition} 
 
 \begin{proof}
 First we consider $d \ge 5.$      Since $\tilde{g}_2(\complet{\Delta})=0$, eq.~\eqref{tilde-g-and-h'} implies that $h''_{d-2}(\complet{\Delta}) = h'_{d-1}(\complet{\Delta}).$  So an argument along the same lines as in \cite[Theorem 5.2]{Novik-Swartz-09:Socles} (but using Lemma \ref{analog-of-Prop4.24} instead of \cite[Proposition 4.24]{Swartz-09}) shows that the link of every nonsingular vertex in $\complet{\Delta}$ is a stacked sphere,  and the result follows.  This argument depends on the fact that  a $(d-1)$-dimensional homology sphere with $d \ge 4$ and $h_{d-2}=h_{d-1}$ is a stacked sphere.  Since  vertex links of $3$-dimensional homology spheres are  $2$-dimensional spheres and $h_1=h_2$ for all two-dimensional spheres, stacked or not, we use a different approach for  $d=4.$
 
 Thus assume $d=4.$  The definition of $g_i$ shows that 
 $$g_2(\Delta) + g_1(\partial \Delta) = g_2(\Delta, \partial \Delta) + g_2(\partial \Delta).$$
 So $\tilde{g}_2(\complet{\Delta})=0$ and (\ref{g-tilde}) imply that $g_2(\Delta, \partial \Delta) + g_2(\partial \Delta) = 6 \tilde{\beta}_1(\Delta) + 4 \tilde{\beta}_2(\Delta).$ Since $\partial \Delta$ is an orientable compact surface $g_2(\partial \Delta)= 3 \tilde{\beta}_1(\partial \Delta)$ and hence,
 $$g_2(\Delta, \partial \Delta) = 6 \tilde{\beta}_1(\Delta)  - 3 \tilde{\beta}_1(\partial \Delta)+ 4 \tilde{\beta}_2(\Delta).$$
 Now, the long exact sequence of the pair $(\Delta, \partial \Delta)$ implies that 
 $$\tilde{\beta}_3(\Delta, \partial \Delta) + \tilde{\beta}_2(\Delta) + \tilde{\beta}_1(\partial \Delta) + \tilde{\beta}_1(\Delta, \partial \Delta) = \tilde{\beta}_2(\partial \Delta) + \tilde{\beta}_2(\Delta, \partial \Delta)+ \tilde{\beta}_1(\Delta) + \tilde{\beta}_0(\partial \Delta).$$
 \noindent
 Poincar\'e-Lefschetz duality applied to $\Delta$ and $\partial \Delta$ gives us
 $$1 + 2 \tilde{\beta}_1(\Delta, \partial \Delta) + \tilde{\beta}_1(\partial \Delta) = 2 \tilde{\beta}_1(\Delta) + 2 \tilde{\beta}_0 (\partial \Delta) + 1.$$
 \noindent 
 Thus,
 $$g_2(\Delta, \partial \Delta) = 6 \tilde{\beta}_1(\Delta, \partial \Delta) - 6 \tilde{\beta}_0(\partial \Delta) + 4 \tilde{\beta}_2(\Delta) = 10 \tilde{\beta}_1(\Delta, \partial \Delta) - 6 \tilde{\beta}_0(\partial \Delta).$$
 \noindent  By Theorem \ref{Murai-Novik inequalities}, $\Delta$ has minimal $g_2$ and $\partial \Delta$ has only one component.  Theorem \ref{local minimal g2} and the fact that triangulations of two-dimensional disks with no interior vertices are stacked spheres with one vertex removed proves that the links of the vertices of $\Delta$ are as claimed.  
  \end{proof}

 \begin{theorem} \label{minimal-tilde-g}
 Let $\Delta$ be a $(d-1)$-dimensional, connected, orientable  homology manifold with boundary such that $\complet{\Delta}$ has minimal $\tilde{g}_2$ and $d \ge 4.$  Then there exists a sequence of $(d-1)$-dimensional homology manifolds $\Delta_1 \longrightarrow \cdots \longrightarrow \Delta_r=\Delta$ such that $\Delta_1$ is a stacked homology manifold, and for all $1 \le j \le r-1,~\Delta_{j+1} = \Delta_j \# \Gamma,$ where $\Gamma$ has minimal $\tilde{g}_2$ and no boundary,  or $\Delta_{j+1} = \Delta^\#_j.$ 
 \end{theorem}
 
 \begin{proof}
As in the proof of Theorem \ref{global minimal g2} we can assume that there is no vertex whose link is the boundary of a $(d-1)$-simplex and continue by induction on the number of missing facets in $\Delta.$ If $\Delta$ has a missing facet, then  Propositions \ref{missing facets} and  \ref{connected sum and handle addition for g-tilde}   allow us to write $\Delta$ as a connected sum or handle addition as required for the induction step.  
 
 In preparation for the base case where $\Delta$ has no missing facets, we first show that if the link of any vertex $w$ has a missing facet $F$, then $\{w\} \cup F$ is a missing facet of $\Delta.$  For this it is sufficient to prove that $F \in \Delta.$  To prove that $F \in \Delta$ we follow Walkup's idea in \cite{Walkup-70} and retriangulate $\complet{\Delta}$ as follows. The previous proposition shows that the link of $w$ in $\complet{\Delta}$ is a stacked sphere.   Remove $w$ from $\complet{\Delta}$ and insert $F.$  The union of $\lk_{\complet{\Delta}} w$ and $F$ consists of two PL-spheres whose intersection is $F.$     Now add two new vertices $x$ and $y$ which cone off these two PL-spheres and call the new complex $\complet{\Delta}'.$  Counting  edges shows that $g_2(\complet{\Delta}') = g_2(\complet{\Delta}) -1.$    This is a contradiction since  $\complet{\Delta}$  has minimal $\tilde{g}_2$ and $\complet{\Delta}'$ is homeomorphic to $\complet{\Delta}$. 
To see that $\complet{\Delta}'$ is homeomorphic to $\complet{\Delta}$ we note that $\st w$ and $\st x \cup \st y$ are homeomorphic since they are both $(d-1)$-dimensional PL-balls.  
 

Now assume  $\Delta$ contains no missing facets. We start by observing that a  stacked sphere which is not the boundary of a simplex contains missing facets. Since no vertex link of $\Delta$ can have a missing facet, the previous proposition implies that  every vertex of $\Delta$ is a boundary vertex and its link is a stacked sphere with one vertex removed.   Hence the link of a vertex $w$ of $\Delta$  can be written as $(\Sm_1 \# \cdots \# \Sm_m) - v,$ where the $\Sm_i$ are boundaries of $(d-1)$-simplices.   Of course, $v$ is $v_0$ --- the vertex added to form the completion of $\Delta.$   
It must be the case that $v$ is in every $\Sm_i.$ Otherwise there would be a missing facet in the link of $w$.  But now the union of (images) of $\Sm_i-v$ ($i=1,\ldots,m$) is a stacking of the link of $w$ which proves that the link of $w$ is a stacked ball.  Since all of the links of vertices of $\Delta$ are stacked balls, $\Delta$ is a stacked homology manifold. Indeed, if $F\in\Delta$ were an interior face of $\Delta$ of codimension $\geq 2$, then for any $w\in F$, $F-w$ would be an interior face of codimension $\geq 2$ of the link of $w$. 
   \end{proof}
 
 \begin{remark} All  $\Delta_i$ in the statement of Theorem \ref{minimal-tilde-g} have a nonempty connected boundary.
 \end{remark}
 \noindent   Theorem \ref{minimal-tilde-g} allows a description of the possible topological types of $\Delta$ such that $\complet{\Delta}$ has minimal $\tilde{g}_2.$  
 \begin{corollary}
   If $\Delta$ is a $(d-1)$-dimensional, connected, orientable  homology manifold such that  $d \ge 4$ and $\complet{\Delta}$ has minimal $\tilde{g}_2,$ then $\|\Delta\|$ is a ball, sphere,  orientable handlebody with boundary, orientable $S^{d-2}$-bundle over $S^1,$  or a connected sum of two or more of these which have a (possibly empty)  connected boundary.
    \end{corollary}
 
\begin{problem}
Is there a similar decomposition for minimal $\tilde{g}_i$ when $i \ge 3?$
\end{problem}

{\small
\bibliography{manifolds-biblio}
\bibliographystyle{plain}
}
\end{document}